\documentclass[11pt,a4paper]{article}

\usepackage{amssymb} 
\usepackage{amsmath}
\usepackage{amsthm}
\usepackage{amsfonts}

\usepackage[T1]{fontenc}
\usepackage[latin1]{inputenc}
\usepackage{lmodern}

\usepackage{geometry}
\usepackage{color}

\usepackage{mathrsfs}
      
\usepackage{microtype}     
\usepackage{hyperref}      	
\usepackage{verbatim}

\usepackage{geometry}

\usepackage{extarrows}

\hypersetup{
    colorlinks=true,
    linkcolor=blue,
    citecolor=blue,      
    urlcolor=blue
}

\swapnumbers
\numberwithin{equation}{section}

\theoremstyle{plain}
\newtheorem{theorem}{Theorem}[section]
\newtheorem{lemma}[theorem]{Lemma}
\newtheorem{corollary}[theorem]{Corollary}

\theoremstyle{definition}
\newtheorem{example}[theorem]{Example}
\newtheorem{definition}[theorem]{Definition}
\newtheorem{remark}[theorem]{Remark}
\newtheorem*{remark*}{Remark}

\DeclareMathOperator{\spt}{spt}
\DeclareMathOperator{\cut}{Cut}
\DeclareMathOperator{\Div}{div}
	
\title{Some geometric inequalities for varifolds on Riemannian manifolds based on monotonicity identities}

\author{Christian Scharrer\thanks{Max Planck Institute for Mathematics, Vivatsgasse 7, 53111 Bonn, Germany. \texttt{scharrer@mpim-bonn.mpg.de}}}

\begin{document} 
\maketitle

\begin{abstract} 
	Using Rauch's comparison theorem, we prove several monotonicity inequalities for Riemannian submanifolds. Our main result is a general Li--Yau inequality which is applicable in any Riemannian manifold whose sectional curvature is bounded above (possibly positive). We show that the monotonicity inequalities can also be used to obtain Simon type diameter bounds, Sobolev inequalities and corresponding isoperimetric inequalities for Riemannian submanifolds with small volume. Moreover, we infer lower diameter bounds for closed minimal submanifolds as corollaries. All the statements are intrinsic in the sense that no embedding of the ambient Riemannian manifold into Euclidean space is needed. Apart from Rauch's comparison theorem, the proofs mainly rely on the first variation formula, thus are valid for general varifolds. 
\end{abstract}

\paragraph{MSC-classes 2010}
		49Q15 (Primary); 53C21 (Secondary).

\section{Introduction}
Many inequalities that relate the mean curvature of submanifolds with other geometric quantities such as the diameter can be obtained in some way from monotonicity identities, which are formulas that can be used to deduce monotonicity of weighted density ratios. In the Euclidean case, these identities are typically proven by testing the first variation formula with certain vector fields. One of the main ingredients in the construction of these vector fields is the inclusion map of the submanifold into the ambient Euclidean space. A key observation in the computations is that its relative divergence equals the dimension of the submanifold. In the Riemannian case, the inclusion map of a submanifold is not a vector field; however one can perform analogous arguments by using the vector field $r \nabla r$, where $r$ is the distance function to a given point (see for instance \textsc{Anderson}~\cite{MR679768}).  Indeed, its relative divergence is not constant but can be bounded below on small geodesic balls by Rauch's comparison theorem (see Lemma~\ref{lem:pre:Rauch_application}). Such an idea revealed to be very fruitful; for instance, it enabled \textsc{Hoffman--Spruck} \cite{MR365424} to derive a Sobolev inequality for Riemannian manifolds. The idea of testing the first variation formula with the vector field $r\nabla r$ in combination with Hessian comparison theorems for the distance function that give a lower bound of the relative divergence was used again in the works of \textsc{Karcher--Wood}~\cite{MR765831} and \textsc{Xin}~\cite{MR851976}. Their resulting monotonicity inequalities imply Liouville type vanishing theorems for harmonic vector bundle valued $p$-forms. Later, the same idea was used by several authors to prove vanishing theorems in various settings, see for instance \textsc{Dong--Wei}~\cite{MR2795324}. The technique was recently applied by \textsc{Mondino--Spadaro}~\cite{MR3611014} to derive an inequality that relates the radius of balls with the volume and area of the boundary. See also \textsc{Nardulli--Osorio Acevedo}~\cite{MR4130849} who used the technique to prove monotonicity inequalities for varifolds on Riemannian manifolds. A weighted monotonicity inequality was obtained by \textsc{Nguyen}~\cite{nguyen2021weighted}.

In the present paper, we apply the described technique to prove Riemannian counterparts of the Euclidean monotonicity inequalities and their consequences from \textsc{Simon}~\cite{MR756417,MR1243525} and \textsc{Allard}~\cite{MR0307015}. Our main result is a general Li--Yau inequality (see Theorem~\ref{thm:Li-Yau_inequality}). We start with a brief introduction to intrinsic varifolds on Riemannian manifolds in Section~\ref{subsec:Riemannian_varifolds}. All our monotonicity inequalities (see Section~\ref{sec:monotonicity_inequalities}) as well as our main theorems (see Section~\ref{subsec:Geometric_inequalities}) are valid for general varifolds. In particular, all our results can be applied to isometrically immersed Riemannian submanifolds (see Example~\ref{ex:submanifolds_are_varifolds}). Indeed, the Li--Yau inequality for ambient manifolds with positive upper bound on the sectional curvature is also new in the smooth case.

\subsection{Varifolds on Riemannian manifolds} \label{subsec:Riemannian_varifolds}
Let $m,n$ be positive integers satisfying $m\leq n$. Given any $n$-dimensional vector space~$V$, we define the \emph{Grassmann manifold} $\mathbf G(V,m)$ to be the set of all $m$-dimensional linear subspaces of $V$. For $V = \mathbf R^n$, we write $\mathbf G(n,m):=\mathbf G(\mathbf R^n,m)$. One can show that $\mathbf G(n,m)$ is a smooth Euclidean submanifold, see for instance \cite[3.2.29(4)]{MR41:1976}. 

Let $(N,g)$ be an $n$-dimensional Riemannian manifold. We denote with $\mathbf G_m(TN)$ the \emph{Grassmann $m$-plane bundle} of the tangent bundle $TN$ of $N$. That is, there exists a map $\pi:\mathbf G_m(TN)\to N$ such that for each $p\in N$, the fibre $\pi^{-1}(p)$ is given by the Grassmannian manifold $\mathbf G(T_pN,m)$. Given any open set $U$ in $N$ and a chart $x:U\to \mathbf R^n$ of $N$, we note that $\pi^{-1}[U]$ is homeomorphically mapped onto an open subset of  $\mathbf R^n\times\mathbf G(n,m)$ via
\begin{equation*}
	\mathbf G_m(TN)\supset \pi^{-1}[U] \xlongrightarrow{(x\circ \pi,\mathrm dx_{\pi})}\mathbf R^n\times \mathbf G(n,m).
\end{equation*}
This turns $\mathbf G_m(TN)$ into a differentiable manifold. We define
\begin{equation*}
	\mathbf G_m(N):=\{(p,T)\in N\times\mathbf G_m(TN):p=\pi(T)\}
\end{equation*}
and note that $\mathbf G_m(N)$ and $\mathbf G_m(TN)$ are homeomorphic. In particular, $\mathbf G_m(N)$ is a locally compact and separable metric space. 

With an $m$-dimensional \emph{varifold} in $N$ we mean a Radon measure $V$ over $\mathbf G_m(N)$. The space of all $m$-dimensional varifolds on $N$ is denoted with $\mathbf V_m(N)$. The \emph{weight} measure $\|V\|$ of a varifold $V$ is defined by
\begin{equation*}
	\|V\|(A) = V\{(p,T)\in \mathbf G_m(N): p\in A\} \qquad \text{whenever $A\subset N$.}
\end{equation*}
It is the push forward measure of the varifold under the projection $\mathbf G_m(N)\to N$. In particular, $\|V\|$ is a Radon measure on $N$ (see \cite[Lemma 2.6]{menne2017novel}). 

The space of compactly supported vector fields on $N$ is denoted with $\mathscr X(N)$. Given any $X\in \mathscr X(N)$, $p\in N$, and $T\in \mathbf G(T_pN,m)$ with orthonormal basis $\{e_1,\ldots,e_m\}$, we let
\begin{equation*}
	\Div_TX(p) = \sum_{i=1}^mg_p(\nabla_{e_i}X(p),e_i),
\end{equation*} 
where $\nabla$ denotes the Levi-Civita connection. Moreover, we denote with $\spt X$ the support of $X$. The \emph{first variation} of a varifold~$V$ is defined as the linear functional
\begin{equation*}
	\delta V:\mathscr X(N) \to \mathbf R, \qquad \delta V(X)=\int \Div_T X(p)\,\mathrm dV(p,T).
\end{equation*}
The total variation $\|\delta V\|$ of $\delta V$ is defined by
\begin{equation*}
	\|\delta V\|(U) = \sup\{\delta V(X):X\in\mathscr X(N),\spt X\subset U, g(X,X)\leq1\}
\end{equation*} 
whenever $U$ is an open subset of $N$, and
\begin{equation*}
	\|\delta V\|(A) = \inf\{\|\delta V\|(U):\text{$U$ is open in $N$, $A\subset U$}\}
\end{equation*}
whenever $A$ is any subset of $N$. 

Finally, we say that $H$ is the  \emph{generalised mean curvature} of $V$ in $(N,g)$, if and only if $H:N\to TN$ is $\|V\|$ measurable, $\|\delta V\|$ is a Radon measure over $N$, there exists a $\|\delta V\|$ measurable map $\eta$ taking values in $TN$ such that $\|\delta V\|$ almost everywhere, $g(\eta,\eta)\leq 1$, and
\begin{equation}\label{eq:intro:var:mean_curvature}
	\delta V(X) = -\int g(X,H)\,\mathrm d\|V\| + \int g(X,\eta)\,\mathrm d\|\delta V\|_{\mathrm{sing}},
\end{equation}
where $\|\delta V\|_{\mathrm{sing}} = \|\delta V\| - \|\delta V\|_{\|V\|}$, and $\|\delta V\|_{\|V\|}$ is the absolutely continuous part of $\|\delta V\|$ with respect to $\|V\|$, see \cite[2.9.1]{MR41:1976}.

It remains to mention that each isometrically immersed Riemannian manifold $M\to N$ can be considered as a varifold in $N$. For more details, see Example~\ref{ex:submanifolds_are_varifolds}. 

\subsection{Notation and definitions}
Suppose $(X,d)$ is a metric space, $\mu$ is a Radon measure on $X$, and $m$ is a positive integer. Denote with $\boldsymbol \alpha(m)$ the volume of the unit ball in $\mathbf R^m$. Given any $p\in X$ and $r>0$, we define the balls
\begin{equation*}
	B_r(p):=\{x\in X:d(p,x)< r\},\qquad \bar B_r(p):=\{x\in X:d(p,x)\leq r\}. 
\end{equation*}
The $m$-dimensional \emph{lower density} $\mathbf \Theta^{m}_*(\mu,p)$ and \emph{upper density} $\mathbf \Theta^{*m}(\mu,p)$ of $\mu$ at $p\in X$ are defined by
\begin{equation*}
	\mathbf \Theta^{m}_*(\mu,p) = \liminf_{r\to 0+}\frac{\mu (\bar B_r(p))}{\boldsymbol\alpha(m)r^m},\qquad \mathbf \Theta^{*m}(\mu,p) = \limsup_{r\to 0+}\frac{\mu (\bar B_r(p))}{\boldsymbol\alpha(m)r^m}.
\end{equation*}
Moreover, if $\mathbf \Theta^{m}_*(\mu,p)=\mathbf \Theta^{*m}(\mu,p)$, we let $\mathbf \Theta^{m}(\mu,p):=\mathbf \Theta^{m}_*(\mu,p)$.
The \emph{support} $\spt \mu$ of the measure $\mu$ is defined by
\begin{equation*}
	\spt \mu := X\setminus \bigcup\{U:\text{$U$ is open in $X$, $\mu(U)=0$}\}.
\end{equation*}	

\begin{definition}
	Suppose $N$ is a Riemannian manifold and $p\in N$.
	
	We say that $U$ is an open \emph{geodesically star-shaped neighbourhood} of $p$ if and only if there exists an open star-shaped neighbourhood $D$ of $0$ in $T_pN$ such that the exponential map $\exp_p:D\to U$ is a diffeomorphism with $\exp_p[D] = U$, and all geodesics emanating from $p$ are length-minimising in $U$.
	
	Similarly, we say that the open ball $B_r(p)$ with radius $r>0$ is a \emph{geodesic ball} if it is a geodesically star-shaped neighbourhood of $p$.
\end{definition}
Given a complete Riemannian manifold $N$, $p\in N$, and writing $r=d(p,\cdot)$, we denote with $\cut(p)$ the cut locus of $p$ in $N$, and define the \emph{radial curvature} $K_r$ on $N\setminus\cut(p)$ to be the restriction of the sectional curvature to all planes that contain the gradient $\nabla r$ of~$r$. Moreover, given any subset $A\subset N$, we denote with $i_p(A)$ the injectivity radius at $p$ in $N$ restricted to the subset $A$. 

Typically, we denote with $|\cdot|_g$ the norm induced by a Riemannian metric $g$. 

\subsection{Geometric inequalities} \label{subsec:Geometric_inequalities}
The following lower diameter bound for closed minimal submanifolds was proven and discussed by \textsc{Xia}~\cite{MR1698433} for the special case where the ambient manifold $N$ is given by the $n$-dimensional real projective space of curvature $1$. To the author's knowledge, little is known in the general case. In particular, the study of sharp lower bounds for different model spaces remains an open problem. The theorem is a direct consequence of Lemma~\ref{lem:dia-pin:lower_diameter_bound} in combination with \eqref{eq:pre:rem:Rauch:lower_bound_zero} and Example~\ref{ex:submanifolds_are_varifolds}.
 
\begin{theorem}[Lower diameter bounds for closed minimal Riemannian submanifolds] \label{thm:low_dia_bds_min}
	Suppose $N$ is a complete Riemannian manifold, $M$ is a closed minimal submanifold of~$N$ with extrinsic diameter $d_{\mathrm{ext}}(M)=\sup_{M\times M} d$, $b>0$, the sectional curvature $K$ of $N$ satisfies $\sup_MK\leq b$, and $p\in M$.
	
	Then, there holds
	\begin{equation}\label{eq:thm:lower_diameter_bound_minimal_submanifold}
		d_{\mathrm{ext}}(M) \geq \min\Bigl\{i_p(N),\frac{\pi}{2\sqrt{b}}\Bigr\}.
	\end{equation}
\end{theorem}

\begin{remark}\label{rem:thm:low_dia_bnd_min_submanifold}
	Notice that if the sectional curvature $K$ of $N$ satisfies $K\leq 0$, then \eqref{eq:thm:lower_diameter_bound_minimal_submanifold} becomes
	\begin{equation}\label{eq:intro:lower_diam_bound_negative_curvature}
		d_{\mathrm{ext}}(M) \geq i_p(N).
	\end{equation}
	If instead, the sectional curvatures of $N$ are pinched between $\frac{1}{4}b$ and $b$, and $N$ is simply connected, then, by a result of \textsc{Klingenberg}~\cite{MR139120},  \eqref{eq:thm:lower_diameter_bound_minimal_submanifold} becomes
	\begin{equation}\label{eq:intro:lower_diam_bound}
		d_{\mathrm{ext}}(M) \geq \frac{\pi}{2\sqrt{b}}.
	\end{equation}
	In the special case where $N=\mathbf RP^n$ is the $n$-dimensional real projective space of curvature~1, the lower bound in~\eqref{eq:intro:lower_diam_bound} is attained if and only if $M$ is totally geodesic, see~\cite{MR1698433}. However, the same does not hold for the unit sphere $N=\mathbf S^7$. Indeed, the complex projective space $\mathbf CP^2(\frac{4}{3})$ of complex dimension 2 and complex sectional curvature~$\frac{4}{3}$ has diameter $\frac{\sqrt{3}\pi}{2}$ and can be isometrically and minimally imbedded into $\mathbf S^7$, see \cite{MR1698433,MR470913}. 
	It remains an interesting open problem to study sharp lower diameter bounds for closed minimal submanifolds in different ambient manifolds, in particular in $\mathbf S^n$.  
\end{remark}

We have the following generalisation of~\eqref{eq:intro:lower_diam_bound_negative_curvature} to asymptotically non-positively curved ambient manifolds as a consequence of Lemma~\ref{lem:dia-pin:asymptotically_negatively_curved} in combination with \eqref{eq:pre:rem:Rauch:lower_bound_zero}, and Example~\ref{ex:submanifolds_are_varifolds}.

\begin{theorem}\label{thm:lower_dia_bds_min_asypmt_neg}
	Suppose $m,n$ are positive integers, $m\leq n$, $N$ is a complete $n$-dimensional Riemannian manifold, $p\in N$, $r = d(p,\cdot)$,  $D(p) := N\setminus(\cut(p)\cup \{p\})$, $0<b\leq1/4$, the radial curvature $K_r$ satisfies $K_r \leq \frac{b}{r^2}$ on $D(p)$, and $M$ is a closed minimal submanifold of $N$ with $p\in M$.
	
	Then, the extrinsic diameter $d_{\mathrm{ext}}(M)$ of $M$ in $N$ can be bounded below by the cut locus distance:
	\begin{equation*}
		d_{\mathrm{ext}}(N) \geq i_p(N).
	\end{equation*}
\end{theorem}

In \cite[Lemma 1.1]{MR1243525}, \textsc{Simon} showed a diameter pinching theorem for closed surfaces in the Euclidean space in terms of their Willmore energy and their area. The upper diameter bound was improved and generalised to Euclidean submanifolds by \textsc{Topping}~\cite{MR2410779} using the Michael--Simon Sobolev inequality~\cite{MR344978}. It was further generalised to Euclidean submanifolds with boundary by \textsc{Menne--Scharrer}~\cite{menne2017novel} leading to a priori diameter bounds for solutions of Plateau's problem. The Riemannian equivalent of Topping's upper diameter bound was proven by \textsc{Wu--Zheng}~\cite{MR2823054} using the Hoffman--Spruck Sobolev inequality~\cite{MR365424}. The following theorem is the Riemannian version for varifolds of Simon's diameter pinching. It is proven in Section~\ref{subsec:proof_diameter_pinching}. In the smooth case it follows from \cite{MR2823054}.

\begin{theorem}[Diameter pinching] \label{thm:diameter_pinching}
	Suppose $n$ is an integer, $n\geq2$, $N$ is a complete $n$-dimensional Riemannian manifold with positive injectivity radius $i>0$, $b\geq0$, the sectional curvature satisfies $K\leq b$, $V\in \mathbf V_2(N)$ has generalised mean curvature $H$, $H$ is square integrable with respect to $\|V\|$, $\|\delta V\|$ is absolutely continuous with respect to~$\|V\|$, $\spt\|V\|$ is compact and connected,
	\begin{equation}\label{eq:thm:dia-pin:normal_mean_curvature}
		H(x)\bot T\qquad \text{for $V$ almost all $(x,T) \in \mathbf G_2(N)$},
	\end{equation}
	and 
	\begin{equation}\label{eq:thm:dia-pin:lower_density_bound}
		\mathbf \Theta^{*2}(\|V\|,x) \geq 1\qquad\text{for $\|V\|$ almost all $x\in N$}.
	\end{equation}
	
	If $\|V\|(N)\leq C_{\eqref{eq:proof:dia-pin:definition_area_bound}}(i,b)$, then
	\begin{equation}\label{eq:thm:dia-pin:lower_bound_Willmore}
		\pi \leq \int|H|_g^2\,\mathrm d\|V\|
	\end{equation}
	and the extrinsic diameter $d_{\mathrm{ext}}(\spt \|V\|)$ of the support of $\|V\|$ is bounded above:
	\begin{equation}\label{eq:thm:dia-pin:upper_diameter_bound}
		d_{\mathrm{ext}}(\spt\|V\|) \leq 2 \sqrt{\|V\|(N)}\left(\sqrt{\int |H|_g^2\,\mathrm d\|V\|} + b\|V\|(N)\right).
	\end{equation}
	Moreover, if $d_{\mathrm{ext}}(\spt\|V\|)<\min\{i,\frac{\pi}{3\sqrt{b}}\}$, then
	\begin{equation}\label{eq:thm:dia-pin:lower_diameter_bound}
		\sqrt{\|V\|(N)\left/\int |H|_g^2\,\mathrm d\|V\|\right.} \leq d_{\mathrm{ext}}(\spt\|V\|).
	\end{equation}	
\end{theorem}

\begin{remark}
	Notice that if the injectivity radius is infinite and the sectional curvatures are non-positive, then both the condition on the area for the upper bound and the condition on the diameter for the lower bound disappear. Indeed, for $i=\infty$ and $b=0$, there holds $C_{\ref{eq:proof:dia-pin:definition_area_bound}}(i,b) = \infty = \min\{i,\frac{\pi}{3\sqrt{b}}\}$. Notice also that if $b=0$, then the diameter bounds are the exact analogous of Simon's Euclidean version \cite[Lemma 1.1]{MR1243525}. In fact, the only difference is the restriction on the area through the injectivity radius in case it is not infinite.
	
	If $N = \mathbf R^n$, then the condition on the generalised mean curvature to bo normal~\eqref{eq:thm:dia-pin:normal_mean_curvature} is satisfied for all integral varifolds, see~\cite[Section 5.8]{MR485012}. Similarly, condition~\eqref{eq:thm:dia-pin:lower_density_bound} is satisfied for all integral varifolds. On the other hand, \eqref{eq:thm:dia-pin:lower_density_bound} implies rectifiability, see~\cite[Theorem 5.5(1)]{MR0307015}.
\end{remark}

In the early 60s, \textsc{Willmore} \cite{MR0202066} showed that the energy now bearing his name is bounded below by $4\pi$ on the class of closed surfaces $\Sigma\subset \mathbf R^3$: 
\begin{equation}\label{eq:intro:Willmore_inequality}
	\frac{1}{4}\int_{\Sigma}H^2\,\mathrm d\mu \geq 4\pi
\end{equation}
where $H$ denotes the sum of the principal curvatures and $\mu$ is the canonical Radon measure on $\Sigma$ given by the immersion. The inequality is also referred to as \emph{Willmore inequality}. Equality holds only for round spheres. Willmore's inequality was improved by \textsc{Li--Yau} \cite[Theorem~6]{MR674407} for smoothly immersed closed surfaces $f:\Sigma\to\mathbf R^n$: If there exists $p\in\mathbf R^n$ with $f^{-1}(p)=\{x_1,\ldots,x_k\}$ where the $x_i$'s are all distinct points in $\Sigma$, in other words $f$ has a point of multiplicity $k$, then 
\begin{equation*}
	\frac{1}{4}\int_{\Sigma}|H|^2\,\mathrm d\mu \geq 4\pi k.
\end{equation*}
In particular, if the Willmore energy $\frac{1}{4}\int_{\Sigma}|H|^2\,\mathrm d\mu$ lies strictly below $8\pi$, then $f$ is an embedding. Because of this property, the Li--Yau inequality has become very useful for the minimisation of the Willmore functional and, more generally, for the study of immersed surfaces. Due to the conformal invariance of the Willmore functional observed by \textsc{Chen}~\cite{MR0370436}, Willmore's inequality has an analogue for surfaces $\Sigma$ in the three sphere~$\mathbf S^3$:
\begin{equation}\label{eq:intro:spherical_Willmore_inequality}
	\frac{1}{4}\int_{\Sigma}H^2\,\mathrm d\mu + |\Sigma| \geq 4\pi
\end{equation} 
and an analogue for surfaces $\Sigma$ in the hyperbolic space $\mathbf H^3$:
\begin{equation}\label{eq:intro:Hyperbolic_Willmore_inequality}
	\frac{1}{4}\int_{\Sigma}H^2\,\mathrm d\mu - |\Sigma| \geq 4\pi,
\end{equation} 
where $|\Sigma|=\int1\,\mathrm d\mu$ denotes the area of $\Sigma$. \textsc{Kleiner}~\cite{MR1156385} showed
\begin{equation}\label{eq:intro:Willmore_inequality_negative_curvature}
	\frac{1}{4}\int_{\Sigma_0} H^2\,\mathrm d\mu + b|\Sigma_0|\geq 4\pi 
\end{equation}
for minimisers $\Sigma_0$ of the isoperimetric profile in a complete one-connected 3-dimensional Riemannian manifold without boundary and with sectional curvatures bounded above by $b\leq0$. Then, \textsc{Ritor\'e} \cite{MR2167269} showed that \eqref{eq:intro:Hyperbolic_Willmore_inequality} remains valid for all $C^{1,1}$ surfaces in a 3-dimensional Cartan--Hadamard manifold with sectional curvatures bounded above by~$-1$. \textsc{Schulze} showed that the classical Willmore inequality~\eqref{eq:intro:Willmore_inequality} holds true for integral 2-varifolds in 3-dimensional Cartan--Hadamard manifolds~\cite[Lemma~6.7]{MR2420018}. Then, he showed that \eqref{eq:intro:Willmore_inequality_negative_curvature} remains valid for integral 2-varifolds in $n$-dimensional Cartan--Hadamard manifolds, see \cite[Theorem~1.4]{MR4080508}. Finally, \textsc{Chai}~\cite{MR4131799} upgraded \eqref{eq:intro:Hyperbolic_Willmore_inequality} to a Li--Yau inequality. That is, given a smoothly immersed surface $f:\Sigma\to\mathbf H^n$ in the hyperbolic space $\mathbf H^n$ which has a point of multiplicity $k$, then
\begin{equation*}
	\frac{1}{4}\int_{\Sigma}|H|^2\,\mathrm d\mu - |\Sigma| \geq 4\pi k.
\end{equation*} 
Notice also the recent generalisation of the Willmore inequality to higher dimensional submanifolds by \textsc{Agostiniani--Fogagnolo--Mazzieri}~\cite{MR4169055}. They showed that for closed codimension 1 submanifolds $M$ in a non-compact $n$-dimensional Riemannian manifold $(N,g)$ with non-negative Ricci curvature, there holds
\begin{equation*}
	\int_{M}\left|\frac{H}{n-1}\right|^{n-1}\,\mathrm d\mu \geq \mathrm{AVG}(g)|\mathbf S^{n-1}|
\end{equation*}
where $\mathrm{AVG}(g)$ denotes the asymptotic volume ratio of $(N,g)$. See also \textsc{Chen}~\cite{MR291994} for the earlier Euclidean version.

In our following theorem, we upgrade the result of \textsc{Schulze}~\cite{MR4080508} about the inequality~\eqref{eq:intro:Willmore_inequality_negative_curvature} to a Li--Yau inequality in any non-positively curved ambient manifold. More importantly, we upgrade the spherical Willmore inequality~\eqref{eq:intro:spherical_Willmore_inequality} to a Li--Yau inequality for varifolds on any Riemannian manifold with an upper bound (possibly positive) on the sectional curvature. The proof is done in Section~\ref{subsec:Li-Yau}. For non-positively curved ambient manifolds, it is analogous to \textsc{Chai}~\cite{MR4131799}; for ambient manifolds with positive upper bound on the curvature it is inspired by \textsc{Simon}~\cite{MR1243525}.

\begin{theorem}[Li--Yau inequality]\label{thm:Li-Yau_inequality}
	Suppose $n$ is an integer, $n\geq2$, $(N,g)$ is an $n$-dimensional Riemannian manifold, $p\in N$, $U$ is a geodesically star-shaped open neighbourhood of $p$, $V\in \mathbf V_2(N)$ has generalised mean curvature $H$, $H$ is square integrable with respect to $\|V\|$, $H(x)\bot T$ for $V$ almost all $(x,T) \in \mathbf G_2(N)$, $p\notin \spt\|\delta V\|_{\mathrm{sing}}$, $\spt \|V\|$ is compact, $\spt\|V\|\subset U$, $b \in \mathbf R$, the sectional curvature of $N$ satisfies $\sup_{\spt\|V\|}K\leq b$, and either 
	\begin{equation*}
		b>0,\qquad \sup_{q\in\spt\|V\|}d(p,q)<\frac{\pi}{2\sqrt{b}},\quad\text{and}\quad C=\frac{16}{\pi^2}
	\end{equation*}
	or 
	\begin{equation*}
		b\leq0\quad\text{and}\quad C=1.
	\end{equation*}
	
	Then, there holds
	\begin{equation*}
		4\pi\mathbf \Theta^2(\|V\|,p) \leq \frac{1}{4} \int_N|H|_g^2\,\mathrm d\|V\| + bC\|V\|(N) + \int_N t_b(r)\,\mathrm d\|\delta V\|_{\mathrm{sing}}
	\end{equation*}
	where $t_b(r) = 2/r$ if $b\geq0$, and $t_b(r) = 2\sqrt{|b|}\coth(\sqrt{|b|}r/2)$ if $b<0$.
\end{theorem}
\begin{remark}
	Notice that the existence of the density $\mathbf \Theta^2(\|V\|,p)$ is part of the statement. Indeed, existence of the density as well as its upper semi-continuity are local statements that do not require any global upper bounds on the curvature nor do they require positive injectivity radius, see Theorem~\ref{thm:mon-ine:existence_density}.
\end{remark}
If the varifold is given by a smoothly immersed surface, then the theorem reads as follows (see Example~\ref{ex:submanifolds_are_varifolds}).
\begin{corollary}
	Suppose $n$ is an integer, $n\geq3$, $(N,g)$ is an $n$-dimensional Riemannian manifold, $b\in\mathbf R$, the sectional curvature $K$ of $N$ satisfies $K\leq b$, $\Sigma$ is a closed surface, $f:\Sigma \to N$ is a smooth immersion, and $f$ has a point of multiplicity $k$. Let $H$ be the trace of the second fundamental form of the immersion $f$, $\mu$ be the Radon measure on $\Sigma$ induced by the pull-back metric of $g$ along $f$, and $|\Sigma| := \int_{\Sigma}1\,\mathrm d\mu$ be the area of $\Sigma$ in $N$. Then, the following two statements hold.
	\begin{enumerate}
		\item\label{itme:Li-Yau:positive_curvature} If $b>0$, and the image of $f$ is contained in a geodesic ball of radius at most $\frac{\pi}{2\sqrt{b}}$, then 
		\begin{equation*}
			\frac{1}{4}\int_{\Sigma}|H|^2_g\,\mathrm d\mu + \frac{16}{\pi^2}b|\Sigma|\geq 4\pi k.
		\end{equation*}
		\item\label{itme:Li-Yau:negative_curvature} If $b\leq0$ and the image of $f$ is contained in a geodesically star-shaped open neighbourhood of the multiplicity point, then
		\begin{equation*}
			\frac{1}{4}\int_{\Sigma}|H|^2_g\,\mathrm d\mu + b|\Sigma|\geq 4\pi k.
		\end{equation*}
	\end{enumerate}	
	In particular, if the left hand side is strictly smaller than $8\pi$, then $f$ is an embedding.
\end{corollary}

\begin{remark}
	The upper bound on the radius of the geodesic ball containing the image of $f$ in \eqref{itme:Li-Yau:positive_curvature} can be enlarged up to $\frac{\pi}{\sqrt{b}}$ at the cost of a constant larger than $\frac{16}{\pi^2}$ in front of the area. Moreover, in view of Lemma~\ref{lem:mon-ine:positive_curvature} and \eqref{eq:rem:mon-ine:pos-cur}, the constant $\frac{16}{\pi^2}$ cannot be expected to be sharp. It is an interesting open question whether or not the constant $\frac{16}{\pi^2}$ can be replaced by 1. In view of the spherical version~\eqref{eq:intro:spherical_Willmore_inequality} this seems possible.
	
	If $N$ is a Cartan--Hadamard manifold, then $N$ itself is a geodesically star-shaped open neighbourhood of any point. In particular, there is no condition on $f$ in~\eqref{itme:Li-Yau:negative_curvature}.
\end{remark}

Isoperimetric inequalities play an important role in the theory of varifolds and its applications, see for instance \cite{MR0307015,MR2537022,menne2017novel}. They can be derived from Sobolev inequalities. Allard~\cite[Theorem~7.1]{MR0307015} proved a Sobolev inequality for general varifolds in Euclidean space with a constant depending on the dimension of the varifold and the dimension of the ambient space. \textsc{Michael--Simon}~\cite{MR344978} proved a Sobolev inequality for generalised submanifolds in Euclidean space where the constant depends only on the dimension of the submanifold. Its proof was later adapted by \textsc{Simon}~\cite[Theorem~18.6]{MR756417} for varifolds whose first variation is absolutely continuous with respect to the weight measure. These varifolds correspond to submanifolds without boundary. \textsc{Menne--Scharrer}~\cite{MR3777387} proved a general Sobolev inequality accounting for the unrectifiable part of the varifold as well. Finally, \textsc{Hoyos}~\cite{hoyos2020intrinsic,hoyos2020poincar} generalised \textsc{Simon}'s version \cite[Theorem~18.6]{MR756417} for varifolds on Riemannian manifolds. Our following theorem generalises his inequality for arbitrary varifolds whose first variation doesn't have to be absolutely continuous with respect to the weight measure. In this way, our resulting isoperimetric inequality indeed recovers the smooth version of \textsc{Hoffman--Spruck}~\cite[Theorem~2.2]{MR365424}. The proof of the following theorem can be found in Section~\ref{subsec:proof:Sobolev_inequality}.

\begin{theorem}[Sobolev inequality]\label{thm:Sobolev_inequality}
	Suppose $m, n$ are positive integers, $m\leq n$, $(N,g)$ is a complete $n$-dimensional Riemannian manifold, $V\in \mathbf V_m(N)$, $\|\delta V\|$ is a Radon measure, $i(\spt\|V\|)$ denotes the injectivity radius of $N$ restricted to $\spt\|V\|$, $b>0$, the sectional curvature $K$ of $N$ satisfies $\sup_{\spt\|V\|}K\leq b$,  $d_{\mathrm{ext}}(\spt\|V\|)$ denotes the extrinsic diameter of $\spt\|V\|$ in $N$, and 
	\begin{equation*}
		\min\{d_{\mathrm{ext}}(\spt\|V\|), (\boldsymbol\alpha(m)^{-1}2^{m+1}\|V\|(N))^{1/m}\}<\min\Bigl\{i(\spt\|V\|),\frac{\pi}{2\sqrt{b}}\Bigr\}.
	\end{equation*}
	 
	Then, for all non-negative compactly supported $C^1$ functions $h\leq1$ on $N$, there holds
	\begin{align*}
		&\int_{\{x\in N:h(x)\mathbf\Theta^{m}(\|V\|,x)\geq1\}}h\,\mathrm d\|V\| \\
		&\quad\leq C_\eqref{eq:Sob-ine:constant}(m)\left(\int h\,\mathrm d \|V\|\right)^{1/m}\left(\int h\,\mathrm d\|\delta V\|+ m\sqrt{b}\int h\,\mathrm d\|V\| + \int |\nabla^\top h|_g\,\mathrm dV\right).
	\end{align*}	
	where $(\nabla^\top h)(x,T)$ is the orthogonal projection of $\nabla h(x)$ onto $T$ for $(x,T)\in\mathbf G_m(N)$.
\end{theorem}

One can absorb the curvature dependent term in the middle to obtain the following isoperimetric inequality. For its proof see Section~\ref{subsec:proof:isoperimetri_inequality}. 

\begin{corollary}[Isoperimetric inequality]\label{cor:sob-ine}
	If in addition  
	\begin{equation}\label{eq:cor:iso-ine:small_area_condition}
		2C_\eqref{eq:Sob-ine:constant}(m)m\sqrt{b}\|V\|(N)^{1/m} \leq 1
	\end{equation}
	and $\mathbf \Theta^{m}(\|V\|,x)\geq 1$ for $\|V\|$ almost all $x\in \spt\|V\|$, then
	\begin{equation*}
		\|V\|(N)^{\frac{m-1}{m}}\leq 2C_\eqref{eq:Sob-ine:constant}(m)\|\delta V\|(N).
	\end{equation*}
\end{corollary}
\begin{remark}
	Notice that by Example~\ref{ex:submanifolds_are_varifolds}, the inequality above indeed implies the isoperimetric inequality  \cite[Theorem~2.2]{MR365424} for smoothly immersed Riemannian manifolds. Moreover, if the sectional curvatures of $N$ are non-positive (for instance if $N$ is Cartan--Hadamard), then the condition on the volume \eqref{eq:cor:iso-ine:small_area_condition} is satisfied for all $V$ with compact support.
\end{remark}

\textbf{Acknowledgements.}
The author is supported by the EPSRC as part of the MASDOC DTC at the University of Warwick, grant No. EP/HO23364/1. Moreover, I would like thank \textsc{Andrea Mondino} for supervising me, as well as \textsc{Lucas Ambrozio} and \textsc{Manh Tien Nguyen} for discussions on the subject.

\section{Preliminaries}
In this section, we show how the first variation can be represented by integration, see Lemma~\ref{lem:pre:representing_first_variation} and Lemma~\ref{lem:pre:Radon_Nikodym}. These are basic facts that have been frequently used in the study of Euclidean varifolds. We are going to need them in order to prove the Sobolev inequality (Theorem~\ref{thm:Sobolev_inequality}). Moreover, we prove that any smoothly immersed manifold is a varifold, see Lemma~\ref{lem:pre:push_forward_under_immersion} and Example~\ref{ex:submanifolds_are_varifolds}. Finally, we state the Hessian comparison theorems for the distance function (see Lemma~\ref{lem:pre:Rauch} and Lemma~\ref{lem:pre:Rauch_application}) that are crucial to derive the monotonicity inequalities in Section~\ref{sec:monotonicity_inequalities}.

\begin{lemma}[\protect{Compare \cite[4.3(2)]{MR0307015}}] \label{lem:pre:representing_first_variation}
	Suppose $m,n$ are positive integers, $m \leq n$, $(N,g)$ is an $n$-dimensional Riemannian manifold, $V\in \mathbf V_m(N)$, and $\|\delta V\|$ is a Radon measure.
	
	Then, there exists a $\|\delta V\|$ measurable map $\eta:N\to TN$ such that $\|\delta V\|$ almost everywhere, $g(\eta,\eta)\leq1$, and
	\begin{equation*}
	\delta V(X) = \int g(X,\eta) \,\mathrm d\|\delta V\|.
	\end{equation*}
\end{lemma}

\begin{remark*}
	Notice that by definition, $\|\delta V\|$ is a Borel regular measure on $N$. Hence, it is a Radon measure if and only if $\|\delta V\|(K)<\infty$ for all compact sets $K \subset N$.
\end{remark*}

\begin{proof}
	Let $\{(x_\alpha,U_\alpha):\alpha \in I\}$ be a countable differentiable atlas on $N$. For all $\alpha \in I$ define 
	\begin{equation*}
		T_\alpha:C_c^\infty(U_\alpha,\mathbf R^n) \to \mathbf R, \qquad T_\alpha(X) = \delta V\Bigl(X^i\frac{\partial}{\partial x_\alpha^i}\Bigr),
	\end{equation*}
	where we've used the Einstein summation convention, and $\frac{\partial}{\partial x_\alpha^i}$ are the coordinate vector fields corresponding to the chart $x_\alpha$. Then,
	\begin{align*}
		&\sup T_\alpha[\{X\in C_c^\infty(U_\alpha,\mathbf R^n):X^iX^jg_{\alpha ij}\leq 1\}] \\
		&\quad = \sup \delta V[\{X\in\mathscr X(N):\spt X\subset U_\alpha,\,g(X,X)\leq 1\}]<\infty.
	\end{align*}
	Therefore, we can apply the representation theorem \cite[2.5.12]{MR41:1976} in combination with Lusin's theorem \cite[2.3.6]{MR41:1976} to obtain a Borel map $k_\alpha$ on $U_\alpha$ taking values in the dual space $(\mathbf R^n)^*$ of $\mathbf R^n$ such that $\|\delta V\|$ almost everywhere, $k_{\alpha i}k_{\alpha j}g_\alpha^{ij} = 1$, and
	\begin{equation*}
		T_\alpha(X) = \int_{U_\alpha}X^i k_{\alpha i}\,\mathrm d\|\delta V\|.
	\end{equation*}
	For $\eta_\alpha:= k_{\alpha i}g_\alpha^{ij}\frac{\partial}{\partial x_\alpha^j}$, this means that $\|\delta V\|$ almost everywhere, $g(\eta_\alpha,\eta_\alpha) = 1$, and
	\begin{equation*}
		\delta V(X) = T_\alpha(X_\alpha) = \int_{U_\alpha}X_\alpha^ik_{\alpha i}\,\mathrm d\|\delta V\|=\int_{U_\alpha}g(X,\eta_\alpha)\,\mathrm d\|\delta V\|
	\end{equation*}
	whenever $X= X_\alpha^i\frac{\partial}{\partial x_\alpha^i} \in \mathscr X(U_\alpha)$. This equation uniquely determines the Borel map~$\eta_\alpha$ up to a set of $\|\delta V\|$ measure zero. Hence, we can define a $\|\delta V\|$ measurable map $\eta:N\to TN$ by requiring that
	\begin{equation*}
		\eta(x) = \eta_\alpha(x)\qquad \text{whenever $\alpha\in I$ and $x\in U_\alpha$}.
	\end{equation*}
	Now, let $X\in\mathscr X(N)$. Choose a finite covering $\{U_\alpha:\alpha\in I_0\}$ of the support of $X$ together with a subordinate partition of unity $\{\varphi_\alpha:\alpha\in I_0\}$. Then,
	\begin{equation*}
		\delta V(X) = \sum_{\alpha\in I_0}\delta V(\varphi_\alpha X) = \sum_{\alpha\in I_0}\int_{U_\alpha}g(\varphi_\alpha X,\eta_\alpha)\,\mathrm d\|\delta V\| = \int g(X,\eta)\,\mathrm d\|\delta V\|
	\end{equation*}
	which finishes the proof.
\end{proof}

\begin{lemma}[\protect{Compare \cite[4.3(5)]{MR0307015}}]\label{lem:pre:Radon_Nikodym}
	Suppose $m,n$ are positive integers, $m\leq n$, $N$ is an $n$-dimensional Riemannian manifold, $V\in \mathbf V_m(N)$, and $\|\delta V\|$ is a Radon measure.
	
	Then, $V$ has a generalised mean curvature (see \eqref{eq:intro:var:mean_curvature}).
\end{lemma}

\begin{proof}
	First, assume that $\spt \|\delta V\|$ is compact. Then, by \cite[2.8.9]{MR41:1976} one can apply the usual theory of symmetrical derivation (i.e. \cite[2.8.18, 2.9.7]{MR41:1976}) to the integral in Lemma~\ref{lem:pre:representing_first_variation}. The general case then follows by covering $\spt \|\delta V\|$ with countably many open sets whose closure is compact.
\end{proof}

The following lemma is the Riemannian counterpart of \cite[Lemma 2.8]{menne2017novel}.
\begin{lemma}\label{lem:pre:push_forward_under_immersion}
	Suppose $m,n$ are positive integers, $m\leq n$, $M$ is a compact $m$-dimensional connected differentiable manifold with boundary, $(N,g)$ is an $n$-dimensional Riemannian manifold, and $f:M\to N$ is a smooth proper immersion. Denote with $\mathcal H^m_g$ the $m$-dimensional Hausdorff measure on $N$ with respect to the distance induced by the metric~$g$, and denote with $\mu_{f^*g}$ the Riemannian measure on $M$ corresponding to the pull back metric $f^*g$ of $g$ along $f$.
	
	Then, there holds 
	\begin{equation} \label{eq:pre:sum_integration}
		\int_{M}k\,\mathrm d\mu_{f^*g} = \int_{N}\sum_{p\in f^{-1}(x)}k(p)\,\mathrm d\mathcal H^m_gx
	\end{equation}
	for all compactly supported continuous functions $k:M\to\mathbf R$. In particular, the push forward measure $f_\#\mu_{f^*g}$ of $\mu_{f^*g}$ under $f$ is a Radon measure on $N$ and satisfies
	\begin{equation} \label{eq:pre:immersion_measure}
		f_\#\mu_{f^*g}(B) = \int_B \mathcal H^0(f^{-1}\{x\})\,\mathrm d\mathcal H^m_g x\qquad \text{for all Borel sets $B\subset N$},
	\end{equation}	
	where $\mathcal H^0$ denotes the counting measure. Moreover, for all $x\in f[M\setminus \partial M]$, there holds
	\begin{equation} \label{eq:pre:density_immersion_measure}
		\mathbf \Theta^m(f_\#\mu_{f^*g},x) = \mathcal H^0(f^{-1}\{x\})
	\end{equation}
	and for $f_\#\mu_{f^*g}$ almost all $x$,
	\begin{equation}\label{eq:pre:tangent_planes_immersion}
		\mathrm df_p[T_pM] = \mathrm df_q[T_qM]\qquad \text{whenever $p,q\in f^{-1}\{x\}$.}
	\end{equation}
\end{lemma}

\begin{proof}
	First, suppose that $\partial M = \varnothing$,  $M$ is a submanifold of $N$ and $f=i$ is the inclusion map. Denote with $d_M$ and $d_N$ the distance functions on $(M,i^*g)$ and $(N,g)$, respectively. Clearly, $d_M(a,b)\geq d_N(a,b)$ for all $a,b \in M$. By \cite[3.2.46]{MR41:1976}, $\mu_{i^*g}$ coincides with the $m$-dimensional Hausdorff measure $\mathcal H^m_{i^*g}$ on $M$ corresponding to the distance $d_M$. Thus, $i_\#\mu_{i^*g}(S) = \mathcal H^m_{i^*g}(M\cap S) \geq \mathcal{H}^m_g(M\cap S)$ for all $S\subset N$. To prove the reverse inequality, let $\lambda>1$ and $p\in M$. Choose an open neighbourhood $U$ of $p$ in $N$ together with a submanifold chart $x:U \to \mathbf R^n$, i.e. $x[M\cap U] = x[U]\cap (\mathbf R^{m}\times\{0\})$. Composing $x$ with a linear map $\mathbf R^m\times \mathbf R^{n-m}\to \mathbf R^m\times \mathbf R^{n-m}$, we may assume that $\mathrm dx_p$ maps an orthonormal basis of $T_pN$ onto an orthonormal basis of $\mathbf R^n$. In particular, $\|\mathrm dx_p\| = \|\mathrm dx^{-1}_{x(p)}\| = 1$. Hence, there exists $\rho>0$ such that $B_\rho(p) := \{q\in N:d_N(p,q)<\rho\} \subset U$, as well as $\|\mathrm dx|_{B_\rho(p)}\| \leq \sqrt{\lambda}$ and $\|\mathrm d(x|_{B_\rho(p)})^{-1}\|\leq \sqrt{\lambda}$. Thus, $x|_{B_\rho(p)}$ is Lipschitz continuous with Lipschitz constant bounded above by $\sqrt{\lambda}$. Next, choose $\rho_0>0$ such that $\rho_0<\rho$ and the convex hull of $x[M\cap B_{\rho_0}(p)]$ in $\mathbf R^m\times\{0\}$ is contained in $x[M\cap B_\rho(p)]$. Given any $a,b \in M\cap B_{\rho_0}(p)$, let 
	\begin{equation*}
		\gamma:[0,1]\to \mathbf R^m\times\{0\},\qquad \gamma(t)=(1-t)x(a) + tx(b).
	\end{equation*}
	Then, $c:=x^{-1}\circ \gamma$ is a smooth curve in $M\cap B_\rho(p)$, connecting $a$ with $b$. Therefore,
	\begin{equation*}
		d_M(a,b) \leq \int_0^1\sqrt{(i^*g)_{c}(\dot c,\dot c)}\,\mathrm dt\leq\sqrt{\lambda} \int_0^1 |\dot\gamma|\,\mathrm dt=\sqrt{\lambda}|x(a) - x(b)|\leq\lambda d_N(a,b).
	\end{equation*}	 
	This implies $\mathcal H^m_{i^*g}(M\cap S) \leq \lambda^m\mathcal{H}^m_g(M\cap S)$ for all $S\subset B_{\rho_0}(p)$. Thus,
	\begin{equation}\label{eq:pre:density_immersion_measure_special_case}
		\lim_{r\to0+}\frac{\mu_{i^*g}(M\cap \{q:d_M(p,q)\leq r\})}{r^m} = \lim_{r\to 0+}\frac{\mathcal H^m_g(M\cap \{q:d_N(p,q)\leq r\})}{r^m}
	\end{equation}
	in the sense that the left hand side exists if and only if the right hand side exists in which case both sides coincide.	Moreover, it follows $i_\#\mu_{i^*g}(B) = \mathcal H^m_{i^*g}(M\cap B) = \mathcal{H}^m_g(M\cap B)$ for all Borel sets $B\subset N$, which proves \eqref{eq:pre:immersion_measure} for the special case.
	
	Next, suppose that $f$ is an embedding and $\partial M =\varnothing$. Then, $f[M]$ is a submanifold of~$N$. Denote with $i:f[M]\to N$ the inclusion map. Then, $f:(M,f^*(i^*g)) \to (f[M],i^*g)$ is an isometry. This means $f_\#\mu_{f^*(i^*g)} = \mu_{i^*g}$ and, by the first case, 
	\begin{equation*}
		f_\#\mu_{f^*g}(B)=i_\#\mu_{i^*g}(B)=\mathcal H^m_g(f[M]\cap B)\qquad \text{for all Borel sets $B\subset N$}.
	\end{equation*}
	Hence, \eqref{eq:pre:immersion_measure} is valid if $\partial M=\varnothing$ and $f$ is an embedding. Moreover, in this case, \eqref{eq:pre:density_immersion_measure} follows from \cite[Chapter II, Corollary 5.5]{MR1390760} in combination with Equation~\eqref{eq:pre:density_immersion_measure_special_case}.
	
	Now, suppose that $f$ is an immersion and $\partial M = \varnothing$. Let $k:M\to \mathbf R$ be a continuous function with compact support $\spt k$ and choose finitely many open sets $U_1,\ldots,U_\Lambda$ whose union contains $\spt k$ such that $f|_{U_\lambda}$ is an embedding for $\lambda = 1,\ldots,\Lambda$. Pick a subordinate partition of unity $\{\varphi_\lambda\}_{\lambda=1}^\Lambda$, i.e. $\sum_{\lambda=1}^\Lambda\varphi_\lambda(p)=1$ for all $p\in\spt k$, and $\spt\varphi_\lambda \subset U_\lambda$ for $\lambda=1,\ldots,\Lambda$. Given any $x\in f[M]$, then, by injectivity of $f|_{U_\lambda}$ for $\lambda =1,\ldots,\Lambda$, the union
	\begin{equation*}
		\bigcup_{p\in f^{-1}\{x\}}\{\lambda: p\in U_\lambda\}
	\end{equation*} 
	is disjoint. In particular, denoting with $\chi_A$ the characteristic function of any given set $A$,
	\begin{equation*}
		\sum_{\lambda = 1}^\Lambda\varphi_\lambda\bigl((f|_{U_\lambda})^{-1}(x)\bigr)k\bigl((f|_{U_\lambda})^{-1}(x)\bigr)\chi_{f[U_\lambda]}(x) = \sum_{p\in f^{-1}\{x\}}\sum_{\lambda = 1}^\Lambda\varphi_\lambda(p)k(p)= \sum_{p\in f^{-1}\{x\}}k(p).
	\end{equation*}
	Hence, using \eqref{eq:pre:immersion_measure} for the special case,
	\begin{align*}
		\int_M k\,\mathrm d\mu_{f^*g} & =\sum_{\lambda = 1}^\Lambda\int_{U_\lambda} \varphi_\lambda \cdot k\,\mathrm d\mu_{f^*g} =\sum_{\lambda = 1}^\Lambda \int_{f[U_\lambda]}  [\varphi_\lambda \circ(f|_{U_\lambda})^{-1}]\cdot [k \circ (f|_{U_\lambda})^{-1}]\,\mathrm d f_\#\mu_{f^*g}\\
		&=\int_N \sum_{p\in f^{-1}\{x\}}k(p) \,\mathrm d\mathcal H^m_g x.
	\end{align*} 
	This proves \eqref{eq:pre:sum_integration} which readily implies \eqref{eq:pre:immersion_measure}. It remains to mention that if $\partial M \neq \varnothing$, then we have that $\partial(\partial M) = \varnothing$ and $f|_{\partial M}$ is an immersion. By the first cases it follows that $\mathcal H^{m-1}_g(f[\partial M]\cap K) < \infty$ for all compact sets $K\subset N$. Thus, $\mathcal H^m_g(f[\partial M]) = 0$.
	
	To prove \eqref{eq:pre:density_immersion_measure}, we first assume that $f$ is an embedding. Then, the statement follows from \cite[Chapter II, Corollary 5.5]{MR1390760} in combination with Equation~\eqref{eq:pre:density_immersion_measure_special_case}. If $f$ is an immersion let $x\in f[M\setminus\partial M]$ and let $p_1,\ldots,p_k$ be distinct points such that $f^{-1}\{x\}=\{p_1,\ldots,p_k\}$. Choose pairwise disjoint open sets $U_1,\ldots,U_k$ such that for $i=1,\ldots,k$ there holds $p_i\in U_i$. Then, for small $r>0$, there holds
	\begin{equation*}
		f_\#\mu_{f^*g}(B_r(x)) = \sum_{i=1}^k(f|_{U_i})_\#\mu_{(f|_{U_i})^*g}(B_r(x)).
	\end{equation*}
	Hence, \eqref{eq:pre:density_immersion_measure} follows from the special case by linearity of the limit operator.
	
	To prove \eqref{eq:pre:tangent_planes_immersion}, suppose $x\in f[M\setminus \partial M]$, $p_1,p_2 \in f^{-1}\{x\}$, and $U_1,U_2$ are disjoint open neighbourhoods of $p_1,p_2$ in $M$, respectively, such that $f|_{U_1},f|_{U_2}$ are embeddings. Assume that $\mathrm df_{p_1}[T_{p_1}M] \neq \mathrm df_{p_2}[T_{p_2}M]$. Then, we can pick a unit vector $v_1\in T_xN$ such that $v_1\in \mathrm df_{p_1}[T_{p_1}M]\setminus\mathrm df_{p_2}[T_{p_2}M]$ and there exists $0<\varepsilon<1$ such that the cone
	\begin{equation*}
		C:=\{w\in T_xN:|rw-v_1|_g\leq\varepsilon\text{ for some $r\in\mathbf R$}\}
	\end{equation*}
	satisfies $C\cap \mathrm df_{p_2}[T_{p_2}M] = \varnothing$. Next, we pick $\varepsilon_1>0$ such that $\exp_{p_1}:B_{\varepsilon_1}\subset T_{p_1}M\to M$ is a diffeomorphism on $B_{\varepsilon_1}:=\{\xi\in T_{p_1}M:(f^*g)_{p_1}(\xi,\xi)<\varepsilon_1^2\}$ and introduce polar coordinates
	\begin{equation*}
		\Theta:(0,\varepsilon_1)\times S^{m-1} \to M, \qquad \Theta(t,u) = \exp_{p_1}(tu),
	\end{equation*}
	where $S^{m-1} := \{\xi\in T_{p_1}M:(f^*g)_{p_1}(\xi,\xi)=1\}$, as well as the density function
	\begin{equation*}
		\theta:(0,\varepsilon_1)\times S^{m-1} \to \mathbf R, \qquad \theta(t,u) = t^{m-1}\sqrt{\det (f^*g)_{ij}(\Theta(t,u))}.
	\end{equation*}
	By \cite[Chapter II, Lemma 5.4]{MR1390760}, there holds $\mu_{\Theta^*(f^*g)} = \theta\mu_{g_0}$, where $g_0$ is the canonical product metric on $(0,\varepsilon_1)\times S^{m-1}$. Hence, for $E:= (\Theta\circ(\mathrm df_{p_1})^{-1})[C]$ and $u_1:=(\mathrm df_{p_1})^{-1}(v_1)$, we have by Fubini's theorem
	\begin{equation*}
		\mu_{f^*g}(E\cap B_\rho(p_1)) = \int_{S^{m-1}\cap B_{\varepsilon}(u_1)}\int_0^\rho\theta(t,u)\,\mathrm dt\,\mathrm d\mu_{S^{m-1}}u
	\end{equation*}
	for all $0<\rho<\varepsilon_1$. Noting that $\theta(t,u) = t^{m-1} + O(t^{m+1})$ as $t\to 0+$, it follows
	\begin{equation*}
		\mathbf \Theta^m(\mu_{f^*g}\llcorner E,p_1) = \frac{\mu_{S^{m-1}}(B_{\varepsilon}(u_1))}{\boldsymbol\alpha(m)m}>0,
	\end{equation*}
	where $\mu_{f^*g}\llcorner E$ denotes the Radon measure on $M$ given by $(\mu_{f^*g}\llcorner E)(B) = \mu_{f^*g}(E\cap B)$ for all Borel sets $B\subset M$.
	Hence, by \eqref{eq:pre:immersion_measure} applied to $f|_{U_1}$, there holds	$\mathbf \Theta^m_*(\mathcal H^m_g\llcorner f[E],x)>0$. Notice that if $\gamma$ is a curve in $E\cup\{p_1\}$ with $\gamma(0)=p_1$, then $(f\circ\gamma)^\cdot(0)\in C\cup\{0\}$. Moreover, we make the following observation. Choose $\rho>0$ such that $B_\rho(x)$ is a geodesic ball around $x$ in $N$. Given any unit vector $v\in T_xN$ and $\delta>0$, we denote with $C(v,\delta)$ the image of the set
	\begin{equation*}
		B_\rho(0)\cap \{w\in T_xN: |rw-v|_g<\delta \text{ for some }r\in\mathbf R\}
	\end{equation*}
	under the exponential map $\exp_x:B_\rho(0)\cap T_xN\to N$. Then, given any smooth curve $\gamma$ in $N$ with 
	\begin{equation*}
		\gamma(0)=x \qquad\text{and}\qquad\frac{\dot\gamma(0)}{|\dot\gamma(0)|_g}=v,
	\end{equation*}
	one can use normal coordinates and differentiability of $\gamma$ to show that for some $t_0>0$, there holds
	\begin{equation*}
		\gamma(t)\in C(v,\delta)\cup\{x\}\qquad\text{for all }-t_0<t<t_0.
	\end{equation*}
	This observation together with compactness of $\bar B_\rho(x)$ can be used to show that for small $\rho>0$, $f[E]\cap f[U_2]\cap B_\rho(x)=\varnothing$. Thus, $\mathbf \Theta^m_*(\mathcal H^m_g\llcorner f[M],x)>1$. Now, the conclusion follows from \cite[2.10.19(5)]{MR41:1976}.
\end{proof}

The following example is the Riemannian counterpart to the Euclidean case \cite[Definition 2.14]{menne2017novel}. Compare also with \cite[Section~2.2]{MR2928715}, where smoothness of $f$ is replaced with $W^{2,2}$-regularity. 
\begin{example}\label{ex:submanifolds_are_varifolds}
	Let $f,M,N$ be as in Lemma~\ref{lem:pre:push_forward_under_immersion}. Define $V\in \mathbf V_m(N)$ by letting
	\begin{equation*}
		V(k) = \int_N \sum_{p\in f^{-1}\{x\}}k(x,\mathrm df_p[T_pM])\,\mathrm d\mathcal H^m_gx 
	\end{equation*}
	for all continuous functions $k:\mathbf G_m(N)\to \mathbf R$ with compact support. In view of Lemma~\ref{lem:pre:push_forward_under_immersion}, we have
	\begin{equation*}
		\|V\| = f_\#\mu_{f^*g}, \qquad \spt\|V\| = \mathrm{closure } f[M].
	\end{equation*}
	In particular, $\spt\|V\| = f[M]$ if $M$ is closed. Moreover, for all $x\in f[M\setminus\partial M]$,
	\begin{equation*}
		\mathbf \Theta^m(\|V\|,x) = \mathcal H^0(f^{-1}\{x\})
	\end{equation*}
	and 
	\begin{equation*}
		\|V\|(N) = \int_M1\,\mathrm d\mu_{f^*g}=|M|.
	\end{equation*}
	
	Identify $\mathrm df_p[T_pM]$ with $T_pM$. Let $NM$ be the normal bundle of the immersion $f$. That is, there exists $\pi:NM\to M$ such that for each $p\in M$, the fibre $\pi^{-1}(p)$ is given by the orthogonal complement of $T_pM$ in $T_{f(p)}N$. Denote with $H_f:M\to NM$ the mean curvature vector field of $f$, i.e. the trace of the second fundamental form (see \cite[Definition 3.1]{MR365424}). Define the $\|V\|$ measurable map $H:N\to TN$ by  
	\begin{equation*}
		H(x) = 
		\begin{cases}
			\frac{1}{\mathbf\Theta^m(\|V\|,x)}\sum_{p\in f^{-1}\{x\}}H_f(p)&\text{if $\mathbf\Theta^m(\|V\|,x)>0$}\\
			0&\text{if $\mathbf\Theta^m(\|V\|,x)=0$}.
		\end{cases}
	\end{equation*}
	Let $X\in\mathscr X(N)$. By a simple computation (see \cite[Lemma 3.2(i)]{MR365424}),
	\begin{equation*}
		\Div_{T_pM}X(x) = -g_{x}(X(x),H_f(p)) + \Div (X\circ f)^t(p)
	\end{equation*}
	whenever $p\in M$ and $f(p)=x$, where $(X\circ f)^t$ denotes the orthogonal projection of $(X\circ f)$ onto the tangent bundle $TM$. Integrating this equation and using Lemma~\ref{lem:pre:push_forward_under_immersion} as well as the usual Divergence Theorem on $M$ (see \cite[Chapter II, Theorem 5.11]{MR1390760}), we infer
	\begin{align*}
		\delta V(X) &=  \int_N\sum_{p\in f^{-1}\{x\}}\Div_{T_pM}X(x)\,\mathrm d\mathcal H^m_gx\\
		&=-\int_N\sum_{p\in f^{-1}\{x\}}g_x(X(x),H_f(p))\,\mathrm d\mathcal H^m_gx +\int_N\sum_{p\in f^{-1}\{x\}}\Div(X\circ f)^t(p)\,\mathrm d\mathcal H^m_gx\\
		&=-\int_Ng(X,H)\mathbf\Theta^{m}(\|V\|,\cdot)\,\mathrm d\mathcal H^m_g + \int_M\Div(X\circ f)^t\,\mathrm d\mu_{f^*g}\\		
		&=-\int_N g(X,H)\,\mathrm d\|V\| + \int_{\partial M} ({f|_{\partial M}}^*g)\bigl((X\circ f)^t,\nu\bigr)\,\mathrm d\mu_{{f|_{\partial M}}^*g},
	\end{align*}
	where $\nu$ is the outward unit normal vector field on $\partial M$. In particular, $V$ has generalised mean curvature $H$, 
	\begin{equation*}
		\|\delta V\|(B) \leq \int_B |H|_g\,\mathrm d\|V\| + \int_B \mathcal H^0((f|_{\partial M})^{-1}\{x\})\,\mathrm d\mathcal H^{m-1}_g
	\end{equation*}
	for all Borel sets $B\subset N$, and by \eqref{eq:pre:tangent_planes_immersion},
	\begin{equation*}
		H(x)\bot T\qquad \text{for $V$ almost all $(x,T)\in\mathbf G_m(N)$}.
	\end{equation*}
	By definition of $H$, we have trivially
	\begin{equation*}
		\int_N|H|_g\,\mathrm d\|V\|\leq\int_N\sum_{p\in f^{-1}\{x\}}|H_f(p)|_g\,\mathrm d\mathcal H^m_gx = \int_M|H_f|_g\,\mathrm d\mu_{f^*g} 
	\end{equation*}
	and 
	\begin{align*}
		\int_N|H|_g^2\,\mathrm d\|V\| &\leq \int_N\frac{1}{\mathcal H^0(f^{-1}\{x\})^2}\Biggl(\sum_{p\in f^{-1}\{x\}}|H_f(p)|_g\Biggr)^2\,\mathrm d\|V\|x\\
		&\leq \int_N\frac{1}{\mathcal H^0(f^{-1}\{x\})}\sum_{p\in f^{-1}\{x\}}|H_f(p)|_g^2\,\mathrm d\|V\|x \\
		&= \int_N\sum_{p\in f^{-1}\{x\}}|H_f(p)|^2_g\,\mathrm d\mathcal H^m_gx = \int_M|H_f|^2_{g}\,\mathrm d\mu_{f^*g}.
	\end{align*}
\end{example}

\begin{lemma}[\protect{See \cite[Theorem 6.4.3]{MR3469435}}]\label{lem:pre:Rauch}
	Suppose $(N,g)$ is a Riemannian manifold, $p\in N$, $U$ is a geodesically star-shaped open neighbourhood of $p$, the metric is represented in geodesic polar coordinates $g = \mathrm dr\otimes \mathrm dr + g_r$ on $U$, $b\in \mathbf R$, the sectional curvature satisfies $K\leq b$ on $U$, and either $b\leq0$ or $U\subset B_{\frac{\pi}{\sqrt{b}}}(p)$. 
	
	Then, the Hessian $\nabla^2r$ of $r$ can be bounded below on $U$ by
	\begin{equation*}
	\nabla^2r\geq 
	\begin{cases}
		\sqrt{b}\cot(\sqrt{b}r)g_r&\text{if $b>0$} \\
		\sqrt{-b}\coth(\sqrt{-b}r)g_r&\text{if $b\leq 0$}.
	\end{cases}
	\end{equation*}
\end{lemma}

\begin{remark} \label{rem:Rauch}
	Define the continuous function
	\begin{equation*}
		a:[0,\pi)\to \mathbf R,\qquad a(x) = x \cot(x) 
	\end{equation*}
	where $a(0)=1$.	Using the series expansion
	\begin{equation} \label{eq:pre:rem:Rauch:series_expansion_cotx}
		\cot(x) = \frac{1}{x} - \frac{x}{3} - \frac{x^3}{45} - \ldots \qquad \text{for $0<x<\pi$}
	\end{equation}
	where all higher order terms are negative, we see that $a$ is strictly decreasing. In particular, since $\cot(\frac{\pi}{3})=\frac{1}{\sqrt{3}}$, we have
	\begin{align}
		\label{eq:pre:rem:Rauch:upper_bound_xcotx}
		x\cot x &\leq 1 && \text{for $0\leq x<\pi$},\\ \label{eq:pre:rem:Rauch:lower_bound_pithird}
		x \cot x&\geq \frac{1}{2} && \text{for $0\leq x\leq\frac{\pi}{3}$},\\ \label{eq:pre:rem:Rauch:lower_bound_zero}
		x \cot x&>0 && \text{for $0\leq x < \frac{\pi}{2}$}.
	\end{align}
\end{remark}

\begin{lemma} \label{lem:pre:Rauch_application}
	Suppose $m,n$ are positive integers, $m\leq n$, $(N,g)$ is an $n$-dimensional Riemannian manifold, $p \in N$, $U$ is a geodesically star-shaped open neighbourhood of $p$, $b>0$, the sectional curvature satisfies $K\leq b$ on $U$, and $U\subset B_{\frac{\pi}{\sqrt{b}}}(p)$.
	
	Then, writing $r = d(p,\cdot)$, there holds 
	\begin{equation*}
		\Div_T(r\nabla r) \geq m \sqrt{b}r\cot(\sqrt{b}r) 
	\end{equation*}
	for all $T\in \mathbf G_m(TU)$.
\end{lemma}
\begin{proof}
	Writing the metric in polar coordinates $g = \mathrm dr\otimes \mathrm dr + g_r$ and using Lemma~\ref{lem:pre:Rauch} in combination with \eqref{eq:pre:rem:Rauch:upper_bound_xcotx}, we compute for $b>0$ 
	\begin{equation*}
		\nabla(r\nabla r) = \mathrm dr\otimes \mathrm dr + r\nabla^2r \geq \mathrm dr\otimes \mathrm dr +  \sqrt{b}r\cot(\sqrt{b}r)g_r \geq \sqrt{b}r\cot(\sqrt{b}r)g.
	\end{equation*}
	Given any $T\in \mathbf G_m(TU)$ with orthonormal basis $\{e_1,\ldots,e_m\}$, it follows
	\begin{equation*}
		\Div_T(r\nabla r) =\sum_{i=1}^m\nabla(r\nabla r)(e_i,e_i) \geq \sqrt{b}r\cot(\sqrt{b}r) \sum_{i=1}^mg(e_i,e_i) = m \sqrt{b}r\cot(\sqrt{b}r)
	\end{equation*}
	which concludes the proof.
\end{proof}

\section{Monotonicity inequalities}\label{sec:monotonicity_inequalities}
In this section, we prove several monotonicity inequalities. The version for 2-dimensional varifolds on general Riemannian manifolds (Lemma~\ref{lem:mon-ine:positive_curvature}) will be used to prove existence and upper semi-continuity of the density (see Theorem~\ref{thm:mon-ine:existence_density} in this section), the diameter pinching (see Theorem~\ref{thm:diameter_pinching} and Section~\ref{subsec:proof_diameter_pinching} for its proof), and part~\eqref{itme:Li-Yau:positive_curvature} of the Li--Yau inequality (see Theorem~\ref{thm:Li-Yau_inequality} and Section~\ref{subsec:Li-Yau} for its proof). The version for 2-dimensional varifolds on non-positively curved manifolds (Lemma~\ref{lemma:mon-ine:negative_curvature}) will be used to prove part~\eqref{itme:Li-Yau:negative_curvature} of the Li--Yau inequality. The version for general $m$-dimensional varifolds (Lemma~\ref{lem:mon-ine:m-dimensional}) is needed to prove the Sobolev inequality (see Theorem~\ref{thm:Sobolev_inequality} and Section~\ref{subsec:proof:Sobolev_inequality} for its proof).

The proof of the following Lemma is based on the ideas of the monotonicity formula in \textsc{Simon}~\cite{MR1243525} in combination with a technique of \textsc{Anderson}~\cite{MR679768}. See also~\cite[Lemma~A.3]{MR3008339} for a proof in the presence of boundary, and \cite{MR3148123} for higher dimensional varifolds.
\begin{lemma}\label{lem:mon-ine:positive_curvature}
	Suppose $n$ is an integer, $n\geq2$, $(N,g)$ is an $n$-dimensional Riemannian manifold, $p\in N$, $V\in \mathbf V_2(N)$ has generalised mean curvature $H$, $H$ is square integrable with respect to $\|V\|$, $H(x)\bot T$ for $V$ almost all $(x,T)\in \mathbf G_2(N)$, $b>0$,  $0<\rho<\frac{\pi}{\sqrt{b}}$, the sectional curvature satisfies $K\leq b$ on $\spt\|V\|\cap B_\rho(p)$, $U$ is a geodesically star-shaped open neighbourhood of $p$, and  $\spt \|V\| \cap \bar B_\rho(p) \subset U$.
	
	Then, writing $r=d(p,\cdot)$, there holds
	\begin{align*}
		\frac{\|V\|B_\sigma(p)}{\sigma^2} & \leq \frac{\|V\|B_\rho(p)}{\rho^2} + \frac{1}{16}\int_{B_\rho(p)\setminus B_\sigma(p)}|H|_g^2\,\mathrm d\|V\| + \int_{B_\rho(p)}\frac{1 - a_b(r)}{r^2}\,\mathrm d\|V\|\\
		&\qquad + \int_{B_\rho(p)}\frac{1}{2r}\,\mathrm d\|\delta V\|_{\mathrm{sing}} + \int_{B_\rho(p)}\frac{r}{2\rho^2}\,\mathrm d\|\delta V\|_{\mathrm{sing}}\\
		&\qquad + \int_{B_\sigma(p)}\frac{|H|_g}{2\sigma}\,\mathrm d\|V\| + \int_{B_{\rho}}\frac{|H|_g}{2\rho}\,\mathrm d\|V\|
	\end{align*}
	for all $0<\sigma<\rho$, where $a_b(r) = \sqrt{b}r\cot(\sqrt{b}r)$.
\end{lemma}

\begin{proof}
	Given any $\sigma < t < \rho$ and any non-negative smooth function $\varphi:\mathbf R\to \mathbf R$ whose support is contained in the interval $(-\infty,1)$, we let $X = \varphi(\frac{r}{t})r\nabla r$ and compute
	\begin{equation*}
		\Div_T X = \varphi\Bigl(\frac{r}{t}\Bigr)\Div_T(r\nabla r) + \dot\varphi\Bigl(\frac{r}{t}\Bigr)\frac{r}{t}|\nabla^Tr|_g^2
	\end{equation*}	
	for all $T\in\mathbf G_2(TN)$, where $\nabla^Tr$ denotes the orthogonal projection of $\nabla r$ onto $T$. We write
	\begin{equation*}
		\nabla^\bot r:\mathbf G_2(N)\to TN,\qquad (\nabla^\bot r)(x,T) = (\nabla r) (x)- (\nabla^Tr)(x),
	\end{equation*}
	and notice that
	\begin{equation*}
		1 = |\nabla r|^2_g = |\nabla^T r|^2_g + |\nabla^\bot r|^2_g. 
	\end{equation*}
	Therefore, testing the first variation equation (see \eqref{eq:intro:var:mean_curvature}) with $X$, we infer by Lemma~\ref{lem:pre:Rauch_application}
	\begin{align*}
		&2\int_N \varphi\Bigl(\frac{r}{t}\Bigr)a_b(r)\,\mathrm d\|V\| + \int_{\mathbf G_2(N)}\dot\varphi\Bigl(\frac{r}{t}\Bigr)\frac{r}{t}\Bigl[1 - |\nabla^\bot r|_g^2\Bigr]\,\mathrm dV \\
		&\qquad \leq -\int_N\varphi\Bigl(\frac{r}{t}\Bigr)g(r\nabla r,H)\,\mathrm d\|V\| + \int_N\varphi\Bigl(\frac{r}{t}\Bigr)g(r\nabla r,\eta)\,\mathrm d\|\delta V\|_{\mathrm{sing}}.
	\end{align*}
	There holds
	\begin{equation*}
		-\frac{d}{\mathrm dt}\Bigl[\frac{1}{t^2}\varphi\Bigl(\frac{r}{t}\Bigr)\Bigr] = \frac{1}{t^3}\Bigl[2\varphi\Bigl(\frac{r}{t}\Bigr) + \dot\varphi\Bigl(\frac{r}{t}\Bigr)\frac{r}{t}\Bigr].
	\end{equation*}
	Hence, adding $\int_N2\varphi(\frac{r}{t})(1 - a_b(r))\,\mathrm d\|V\|$ on both sides of the inequality and multiplying with $\frac{1}{t^3}$, it follows
	\begin{equation}\label{eq:mon-ine:pos-cur:before_integrating}
		\begin{split} 
			&-\frac{d}{\mathrm dt}\int\frac{1}{t^2}\varphi\Big(\frac{r}{t}\Big)\,\mathrm d\|V\| - \int_{\mathbf G_2(N)}\dot\varphi\Bigl(\frac{r}{t}\Bigr)\frac{r}{t^4}|\nabla^\bot r|_g^2\,\mathrm dV \\ 
			&\qquad \leq 2\int_N\varphi\Bigl(\frac{r}{t}\Bigr)\frac{1 - a_b(r)}{t^3}\,\mathrm d\|V\|  -\int_N\varphi\Bigl(\frac{r}{t}\Bigr)\frac{g(r\nabla r,H)}{t^3}\,\mathrm d\|V\| \\
			&\quad\qquad+ \int_N\varphi\Bigl(\frac{r}{t}\Bigr)\frac{g(r\nabla r,\eta)}{t^3}\,\mathrm d\|\delta V\|_{\mathrm{sing}}.
		\end{split}
	\end{equation}
	Given any $0<\lambda < 1$, choose $\varphi$ such that $\dot \varphi \leq 0$, $\varphi(s)=1$ for all $s\leq \lambda$, and $\varphi(s) = 0$ for all $s\geq1$. In other words, $\varphi$ approaches the characteristic function of the open interval $(-\infty,1)$ from below as $\lambda \to 1$. In particular, if $\dot\varphi(\frac{r}{t}) \neq 0$, then $\frac{r}{t}\geq \lambda$. Hence,
	\begin{equation} \label{eq:mon-ine:pos-cur:swapping_limit}
		\begin{split} 
			\int_{\mathbf G_2(N)}\dot\varphi\Bigl(\frac{r}{t}\Bigr)\frac{r}{t^4}|\nabla^\bot r|_g^2\,\mathrm dV & \leq \int_{\mathbf G_2(N)}\dot\varphi\Bigl(\frac{r}{t}\Bigr)\frac{r}{t^2}\frac{\lambda^2}{r^2}|\nabla^\bot r|_g^2\,\mathrm dV\\
			& = -\frac{d}{\mathrm dt} \int_{\mathbf G_2(N)}\varphi\Bigl(\frac{r}{t}\Bigr)\frac{\lambda^2}{r^2}|\nabla^\bot r|_g^2\,\mathrm dV.
		\end{split}
	\end{equation}
	Moreover, given any $\|V\|$ integrable real valued function $f$, one computes using Fubini's theorem, writing $r_\sigma := \max\{\sigma,r\}$, and denoting with $\chi_A$ the characteristic function of any set $A$,
	\begin{equation}\label{eq:mon-ine:pos-cur:Fubini} 
		\begin{split} 
			&\int_\sigma^\rho\int_{B_t(p)}\frac{f(x)}{t^3}\,\mathrm d\|V\|x\,\mathrm dt=\int_{B_\rho(p)}\int_\sigma^\rho\frac{f(x)}{t^3}\chi_{\{r<t\}}(x)\,\mathrm dt\,\mathrm d\|V\|x\\ 
			&\qquad=\int_{B_\rho(p)}f(x)\int_{r_\sigma(x)}^\rho\frac{1}{t^3}\,\mathrm dt\mathrm d\|V\|x=\frac{1}{2}\int_{B_\rho(p)}\Bigl(\frac{1}{r_\sigma^2}- \frac{1}{\rho^2}\Bigr)f\,\mathrm d\|V\|
		\end{split}
	\end{equation}
	Therefore, putting \eqref{eq:mon-ine:pos-cur:swapping_limit} into \eqref{eq:mon-ine:pos-cur:before_integrating}, integrating with respect to $t$ from $\sigma$ to $\rho$, letting $\lambda\to1$, and using \eqref{eq:mon-ine:pos-cur:Fubini}, we infer
	\begin{align*}
		\frac{\|V\|B_\sigma(p)}{\sigma^2}&\leq\frac{\|V\|B_\rho(p)}{\rho^2}-\int_{\pi^{-1}[B_\rho(p)\setminus B_\sigma(p)]}\frac{|\nabla^\bot r|_g^2}{r^2}\,\mathrm dV \\
		&\quad + \int_{B_\rho(p)}\Bigl(\frac{1}{r_\sigma^2} - \frac{1}{\rho^2}\Bigr)(1 - a_b(r))\,\mathrm d\|V\|-\frac{1}{2}\int_{B_\rho(p)}\Bigl(\frac{1}{r_\sigma^2} - \frac{1}{\rho^2}\Bigr)g(r\nabla r, H)\,\mathrm d\|V\|\\
		&\quad + \frac{1}{2}\int_{B_\rho(p)}\Bigl(\frac{1}{r_\sigma^2} - \frac{1}{\rho^2}\Bigr)g(r\nabla r, \eta)\,\mathrm d\|\delta V\|_{\mathrm{sing}},
	\end{align*}
	where $\pi:\mathbf G_2(N) \to N$ denotes the canonical projection. Observe that 
	\begin{equation*}
		\Bigl|\frac{1}{4}H + \frac{\nabla^\bot r}{r}\Bigr|_g^2 = \frac{1}{2r}g(\nabla r,H) + \frac{|\nabla^\bot r|^2_g}{r^2} + \frac{1}{16}|H|_g^2.
	\end{equation*}
	Thus, it follows
	\begin{align*}
		&\frac{\|V\|B_\sigma(p)}{\sigma^2}\leq\frac{\|V\|B_\rho(p)}{\rho^2} - \int_{\pi^{-1}[B_\rho(p)\setminus B_\sigma(p)]}\Bigl|\frac{1}{4}H + \frac{\nabla^\bot r}{r}\Bigr|_g^2\,\mathrm dV + \frac{1}{16}\int_{B_\rho(p)\setminus B_\sigma(p)}|H|^2_g\,\mathrm d\|V\|\\
		&\qquad +\int_{B_\sigma(p)}\frac{1-a_b(r)}{\sigma^2}\,\mathrm d\|V\| + \int_{B_\rho(p)\setminus B_\sigma(p)}\frac{1-a_b(r)}{r^2}\,\mathrm d\|V\| - \int_{B_\rho(p)}\frac{1-a_b(r)}{\rho^2}\,\mathrm d\|V\|\\
		&\qquad - \int_{B_\sigma(p)}\frac{g(r \nabla r,H)}{2\sigma^2}\,\mathrm d\|V\| + \int_{B_\rho(p)}\frac{g(r\nabla r,H)}{2\rho^2}\,\mathrm d\|V\| \\
		&\qquad + \int_{B_\sigma(p)}\frac{r}{2\sigma^2}\,\mathrm\|\delta V\|_{\mathrm{sing}} + \int_{B_\rho(p)\setminus B_\sigma(p)}\frac{1}{2r}\,\mathrm d\|\delta V\|_{\mathrm{sing}} + \int_{B_\rho(p)}\frac{r}{2\rho^2}\,\mathrm d\|\delta V\|_{\mathrm{sing}}
	\end{align*}
	which, in view of \eqref{eq:pre:rem:Rauch:upper_bound_xcotx}, implies the conclusion.
\end{proof}

\begin{remark}\label{rem:mon-ine:pos-cur}
	In the above proof, we let $\varphi$ approach the characteristic function of the interval $(-\infty,1)$ from below. If we instead let $\varphi$ approach the characteristic function of the interval $(-\infty,1]$ from above, we obtain the following version for closed balls:
	\begin{equation}\label{eq:rem:mon-ine:pos-cur:closed_balls}
		\begin{split} 
			\frac{\|V\|\bar B_\sigma(p)}{\sigma^2} & \leq \frac{\|V\|\bar B_\rho(p)}{\rho^2} + \frac{1}{16}\int_{\bar B_\rho(p)}|H|_g^2\,\mathrm d\|V\| + \int_{\bar B_\rho(p)}\frac{1 - a_b(r)}{r^2}\,\mathrm d\|V\|\\
			&\qquad + \int_{\bar B_\rho(p)}\frac{1}{r}\,\mathrm d\|\delta V\|_{\mathrm{sing}}\\
			&\qquad + \int_{\bar B_\sigma(p)}\frac{|H|_g}{2\sigma}\,\mathrm d\|V\| + \int_{\bar B_{\rho}}\frac{|H|_g}{2\rho}\,\mathrm d\|V\|
		\end{split}
	\end{equation}
	for all $0<\sigma<\rho$.
	
	Define the function 
	\begin{equation*}
		c:[0,\pi) \to \mathbf R,\qquad c(x) = \frac{1 - x\cot x}{x^2}.
	\end{equation*}
	Then, using the series expansion of $\cot(x)$ (see~\eqref{eq:pre:rem:Rauch:series_expansion_cotx}), we obtain the series expansion for~$c$:
	\begin{equation*}
		c(x) = \frac{1}{3} + \frac{x^2}{45} + \ldots 
	\end{equation*}
	with all higher order terms being positive. In particular, $c(0)=\frac{1}{3}$ and $c$ is strictly increasing. Since $c(\frac{\pi}{2}) = \frac{4}{\pi^2}$, the curvature depending term in Lemma~\ref{lem:mon-ine:positive_curvature} can be estimated by
	\begin{equation} \label{eq:rem:mon-ine:pos-cur}
		\int_{B_\rho(p)}\frac{1-a_b(r)}{r^2}\,\mathrm d\|V\|\leq \frac{4}{\pi^2} b \|V\|(B_\rho(p))\leq b\|V\|(B_\rho(p))
	\end{equation}
	whenever $0<\rho < \frac{\pi}{2\sqrt{b}}$. Analogously for closed balls:
	\begin{equation} \label{eq:rem:mon-ine:pos-cur:bound_integral_closed_balls}
		\int_{\bar B_\rho(p)}\frac{1-a_b(r)}{r^2}\,\mathrm d\|V\|\leq \frac{4}{\pi^2} b \|V\|(\bar B_\rho(p))\leq b\|V\|(\bar B_\rho(p))
	\end{equation}
	whenever $0<\rho < \frac{\pi}{2\sqrt{b}}$.
\end{remark}

The following lemma is the Riemannian analogue of Corollary~4.5 and Remark~4.6 in~\cite{MR3528825}.
\begin{lemma}\label{lem:mon-ine:m-dimensional}
	Suppose $m,n$ are positive integers, $m\leq n$, $N$ is an $n$-dimensional Riemannian manifold, $p\in N$, $r=d(p,\cdot)$, $\rho_0>0$, $a$ is a real valued Borel function on~$B_{\rho_0}(p)$, 
	\begin{equation*}
		\Div_T(r\nabla r)\geq ma(r) \qquad \text{for all $T\in \mathbf G_m(TB_{\rho_0}(p))$},
	\end{equation*}
	$V\in \mathbf V_m(N)$, and $\|\delta V\|$ is a Radon measure.
	
	Then, there holds
	\begin{equation*}
		\frac{\|V\|(\bar B_\sigma(p))}{\sigma^m} \leq \frac{\|V\|(\bar B_\rho(p))}{\rho^m} + \int_\sigma^\rho\frac{1}{t^m}\left(\|\delta V\|(\bar B_t(p)) + m\int_{\bar B_t(p)} \frac{1-a}{r}\,\mathrm d\|V\|\right)\,\mathrm dt
	\end{equation*}
	for all $0<\sigma<\rho<\rho_0$.
\end{lemma}

\begin{proof}
	We proceed similarly as in the proof of Lemma~\ref{lem:mon-ine:positive_curvature}.	Given any $\sigma < t < \rho$ and any non-negative non-decreasing smooth function $\varphi:\mathbf R\to \mathbf R$, we let $X = \varphi(\frac{r}{t})r\nabla r$ and compute
	\begin{equation*}
		\Div_T X = \varphi\Bigl(\frac{r}{t}\Bigr)\Div_T(r\nabla r) + \dot\varphi\Bigl(\frac{r}{t}\Bigr)\frac{r}{t}|\nabla^Tr|_g^2
	\end{equation*}	
	for all $T\in\mathbf G_2(TN)$, where $\nabla^Tr$ denotes the orthogonal projection of $\nabla r$ onto $T$. Writing
	\begin{equation*}
		\nabla^\bot r:\mathbf G_2(N)\to TN,\qquad (\nabla^\bot r)(x,T) = (\nabla r) (x)- (\nabla^Tr)(x),
	\end{equation*}
	there holds
	\begin{equation*}
		\delta V(X)\geq m\int_Na(r)\varphi\Bigl(\frac{r}{t}\Bigr)\,\mathrm d\|V\| + \int_{\mathbf G_m(N)}\dot\varphi\Bigl(\frac{r}{t}\Bigr)\frac{r}{t}\Bigl[1-|\nabla^\bot r|_g^2\Bigr]\,\mathrm dV.
	\end{equation*}
	Adding the term $m\int_N(1-a(r))\varphi(\frac{r}{t})\,\mathrm d\|V\|$ on both sides of the inequality, multiplying by $\frac{1}{t^{m+1}}$, and neglecting positive terms on the right hand side, we infer
	\begin{align*}
		&\frac{\delta V(X)}{t^{m+1}} + \frac{m}{t^m}\int_N\frac{1 - a(r)}{t}\varphi\Bigl(\frac{r}{t}\Bigr)\,\mathrm d\|V\|\\
		&\qquad \geq \frac{m}{t^{m+1}}\int_N\varphi\Bigl(\frac{r}{t}\Bigr)\,\mathrm d\|V\| + \frac{1}{t^m}\int_N\dot\varphi\Bigl(\frac{r}{t}\Bigr)\frac{r}{t^2}\,\mathrm d\|V\| = -\frac{d}{\mathrm dt}\left(\frac{1}{t^m}\int_N\varphi\Bigl(\frac{r}{t}\Bigr)\,\mathrm d\|V\|\right).
	\end{align*}
	Integrating the inequality with respect to $t$ from $\sigma$ to $\rho$, letting $\varphi$ approach the characteristic function of the interval $(-\infty,1]$ from above and using Lemma~\ref{lem:pre:representing_first_variation}, it follows
	\begin{equation*}
		\frac{\|V\|\bar B_\sigma(p)}{\sigma^m}\leq \frac{\|V\|\bar B_\rho(p)}{\rho^m} + \int_\sigma^\rho\frac{1}{t^m}\left(\int_{\bar B_t(p)}\frac{g(r\nabla r,\eta)}{t}\,\mathrm d\|\delta V\| + m\int_{\bar B_t(p)}\frac{1 - a(r)}{t}\,\mathrm d\|V\|\right)\mathrm dt
	\end{equation*}
	which implies the conclusion.
\end{proof}

\begin{remark}
	Suppose $b>0$ and the function $a$ is given by $a(r) = \sqrt{b}r\cot(\sqrt{b}r)$ for $0<r<\frac{\pi}{2\sqrt{b}}$. Define the function
	\begin{equation*}
		c:[0,\pi)\to \mathbf R,\qquad c(x) = \frac{1 - x\cot x}{x}
	\end{equation*}
	Then,using the series expansion of $\cot(x)$ (see \eqref{eq:pre:rem:Rauch:series_expansion_cotx}) we obtain the series expansion for~$c$:
	\begin{equation*}
		c(x) = \frac{x}{3} + \frac{x^3}{45} + \ldots 
	\end{equation*}
	with all higher order terms being positive. In particular, $c(0) = 0$ and $c$ is strictly increasing. Since $c(\frac{\pi}{2}) = \frac{2}{\pi}$, the curvature depending term in Lemma~\ref{lem:mon-ine:m-dimensional} can be estimated by
	\begin{equation}\label{eq:mon-ine:rem:m-dim}
		m\int_{\bar B_t(p)}\frac{1 - a(r)}{r}\,\mathrm d\|V\|\leq m\sqrt{b}\|V\|\bar B_t(p)
	\end{equation}
	whenever $0<t<\frac{\pi}{2\sqrt{b}}$.
\end{remark}

The following lemma is a consequence of Lemma~\ref{lem:mon-ine:m-dimensional}. It can also be derived directly from the first variation formula, see \cite[Theorem 5.5]{scharrer:MSc}.
\begin{lemma} \label{lem:mon-ine:point_measure_zero}
	Suppose $m,n$ are integers, $2\leq m \leq n$, $N$ is an $n$-dimensional Riemannian manifold, $V \in \mathbf V_m(N)$, and $\|\delta V\|$ is a Radon measure. 
	
	Then, there holds $\|V\|\{p\} = 0$ for all $p\in N$.
\end{lemma}

\begin{proof}
	Similarly as in \eqref{eq:mon-ine:pos-cur:Fubini}, one can use Fubini's theorem to compute
	\begin{equation*}
		\int_\sigma^\rho\frac{1}{t^m}\|\delta V\|(\bar B_t(p))\,\mathrm dt = \frac{1}{m-1}\int_{\bar B_\rho(p)} \frac{1}{r_\sigma^{m-1}} - \frac{1}{\rho^{m-1}}\,\mathrm d\|\delta V\| \leq\int_{\bar B_\rho(p)}\frac{1}{r_\sigma^{m-1}}\,\mathrm d\|\delta V\|,
	\end{equation*} 	
	where $r_\sigma = \max\{r,\sigma\}$, and $r = d(p,\cdot)$. For small $\rho>0$ we can apply Lemma~\ref{lem:pre:Rauch_application} for some $b>0$ in combination with Lemma~\ref{lem:pre:representing_first_variation}, Lemma~\ref{lem:mon-ine:m-dimensional}, and \eqref{eq:rem:mon-ine:pos-cur:bound_integral_closed_balls} to infer
	\begin{equation*}
		\frac{\|V\|\bar B_\sigma(p)}{\sigma^m} \leq \frac{\|V\|\bar B_\rho(p)}{\rho^m} + \int_{\bar B_\rho(p)}\frac{1}{r_\sigma^{m-1}}\,\mathrm d\|\delta V\| + m\int_\sigma^\rho\frac{1}{t^{m-1}} b \|V\|(\bar B_t(p))\,\mathrm dt.
	\end{equation*}
	Multiplying with $\sigma^{m-1}$ implies
	\begin{equation*}
		\frac{\|V\|\bar B_\sigma(p)}{\sigma} \leq \frac{\|V\|\bar B_\rho(p)}{\rho} + \|\delta V\|(\bar B_\rho(p)) + m\rho b \|V\|(\bar B_\rho(p))<\infty.
	\end{equation*}
	Now, only the left hand side depends on $\sigma$. Hence, $\mathbf \Theta^{*1}(\|V\|,p)<\infty$ which implies $\|V\|\{p\}=0$.
\end{proof}

The proof of the following theorem is based on the Euclidean version in the appendix of~\cite{MR2119722}. See also \cite[Corollary~5.8]{MR3524220} for the existence of the density. 
\begin{theorem}\label{thm:mon-ine:existence_density}
	Suppose $n$ is an integer, $n\geq2$, $N$ is an $n$-dimensional Riemannian manifold, $V\in\mathbf V_2(N)$ has generalised mean curvature $H$, $H(x) \bot T$ for $V$ almost all $(x,T)\in\mathbf G_2(N)$, and $H$ is locally square integrable with respect to $\|V\|$.
	
	Then, for all $p\in N\setminus \spt\|\delta V\|_{\mathrm{sing}}$, there holds:	
	\begin{enumerate}
		\item \label{itme:mon-ine:ex-den:density} The density $\mathbf \Theta^2(\|V\|,p)$ exists and is finite.
		\item \label{itme:mon-ine:ex-den:semi-continuity} The function $\mathbf\Theta^2(\|V\|,\cdot)$ is upper semi-continuous at $p$.
	\end{enumerate}
\end{theorem}

\begin{proof}
	By \eqref{eq:rem:mon-ine:pos-cur:bound_integral_closed_balls} in combination with Lemma~\ref{lem:mon-ine:point_measure_zero}, we have
	\begin{equation}\label{eq:mon-ine:ex-den:decay_area}
		\int_{\bar B_\rho(p)}\frac{1 - a_b(r)}{r^2}\,\mathrm d\|V\| = o(1) \qquad\text{as $\rho \to 0$},
	\end{equation}
	where $a_b$ is defined as in Lemma~\ref{lem:mon-ine:positive_curvature}.	We abbreviate
	\begin{equation*}
		W(t):= \int_{\bar B_t(p)}|H|^2_g\,\mathrm d\|V\|\quad \text{and}\quad A(t):=\frac{\|V\|\bar B_t(p)}{t^2}
	\end{equation*}
	for $t>0$. Using H\"older's inequality, we deduce
	\begin{equation} \label{eq:mon-ine:ex-den:estimate_total_mean_curvature}
		\int_{\bar B_t(p)} \frac{|H|_g}{2t}\,\mathrm d\|V\|\leq \sqrt{A(t)}\sqrt{W(t)}\leq (1+A(t))\sqrt{W(t)}.
	\end{equation}
	Moreover, since $H$ is locally square integrable,
	\begin{equation} \label{eq:mon-ine:ex-den:decay_Willmore}
		W(t) = o(1) \qquad\text{as $t\to 0$}.
	\end{equation}
	Choose $\rho_0>0$ small enough such that $B_{\rho_0}(p)\cap\spt\|\delta V\|_{\mathrm{sing}}=\varnothing$ and such that Lemma~\ref{lem:mon-ine:positive_curvature} can be applied for some $b>0$. Then, there holds
	\begin{equation*}
		\int_{\bar B_\rho(p)}\frac{1}{2r}\,\mathrm d\|\delta V\|_{\mathrm{sing}} = \int_{\bar B_\rho(p)}\frac{r}{2\rho^2}\,\mathrm d\|\delta V\|_{\mathrm{sing}} = 0
	\end{equation*}
	for all $0<\rho<\rho_0$. Hence,	putting \eqref{eq:mon-ine:ex-den:decay_area}, \eqref{eq:mon-ine:ex-den:estimate_total_mean_curvature}, and \eqref{eq:mon-ine:ex-den:decay_Willmore} into \eqref{eq:rem:mon-ine:pos-cur:closed_balls}, we infer
	\begin{equation*}
		(1 - o_\sigma(1))A(\sigma) - o_\sigma(1) \leq (1+o_\rho(1))A(\rho) + o_\rho(1)
	\end{equation*}
	for all $0<\sigma<\rho<\rho_0$. Applying $\limsup_{\sigma\to0+}$ on the left and $\liminf_{\rho\to0+}$ on the right, it follows
	\begin{equation*}
		\mathbf \Theta^{*2}(\|V\|,p)\leq \mathbf \Theta^{2}_*(\|V\|,p)
	\end{equation*}
	which proves \eqref{itme:mon-ine:ex-den:density}. Hence, letting $\sigma \to0$ in \eqref{eq:rem:mon-ine:pos-cur:closed_balls}, and using $2\sqrt{A(t)}\sqrt{W(t)}\leq A(t) + W(t)$, we have
	\begin{equation*}
		\pi\mathbf \Theta^{2}(\|V\|,p)\leq \frac{\|V\|\bar B_\rho(p)}{\rho^2} + W(\rho) + b \|V\|(\bar B_\rho(p)) + A(\rho) + W(\rho)
	\end{equation*}
	for small $0<\rho<\rho_0$. It follows
	\begin{equation}\label{eq:mon-ine:ex-den:limsup_lower}
		\begin{split}
			&\limsup_{q\to p}\frac{\|V\|\bar B_\rho(q)}{\rho^2} \\ 
			&\quad \geq \limsup_{q\to p}\pi\mathbf\Theta^2(\|V\|,q) - W(2\rho) - b \|V\|(\bar B_{2\rho}(p)) - A(2\rho) - W(2\rho) \\
			&\quad =\limsup_{q\to p}\pi\mathbf\Theta^2(\|V\|,q) - o_\rho(1).
		\end{split}
	\end{equation}
	On the other hand,
	\begin{equation} \label{eq:mon-ine:ex-den:limsup_upper}
		\limsup_{q\to p}\frac{\|V\|\bar B_\rho(q)}{\rho^2}\leq \lim_{\varepsilon \to0}\frac{\|V\|\bar B_{\rho+\varepsilon}(p)}{\rho^2} = \frac{\|V\|\bar B_\rho(p)}{\rho^2}
	\end{equation}
	where we used the limit formula for the measure of decreasing sets (see \cite[2.1.3(5)]{MR41:1976}). Putting \eqref{eq:mon-ine:ex-den:limsup_lower} and \eqref{eq:mon-ine:ex-den:limsup_upper} together and taking the limit $\rho\to0$ implies statement~\eqref{itme:mon-ine:ex-den:semi-continuity}.
\end{proof}

The following lemma is a generalisation of~\cite[Equation~(10)]{MR4131799}.
\begin{lemma}\label{lemma:mon-ine:negative_curvature}
	Suppose $n$ is an integer, $n\geq2$, $(N,g)$ is an $n$-dimensional Riemannian manifold, $p\in N$, $V\in\mathbf V_2(N)$ has generalised mean curvature $H$, $H$ is square integrable with respect to $\|V\|$, $H(x)\bot T$ for $V$ almost all $(x,T)\in\mathbf G_2(N)$, $b<0$, $\rho_0>0$, the sectional curvature satisfies $K\leq b$ on $\spt \|V\|\cap \bar B_{\rho_0}(p)$, $U$ is a geodesically star-shaped open neighbourhood of $p$, and $\spt \|V\| \cap \bar B_{\rho_0}(p) \subset U$. Define the functions
	\begin{equation*}
		s_b:(0,\infty)\to\mathbf R,\qquad s_b(t) = \frac{\sinh(\sqrt{|b|}t)}{\sqrt{|b|}},
	\end{equation*}
	$c_b := s_b'$, and $\phi:=|b|/(c_b-1)$.
	
	Then, writing $r=d(p,\cdot)$, there holds
	\begin{align*}
		&2\phi(\sigma)\int_{\bar B_\sigma(p)}c_b(r)\,\mathrm d\|V\| +  |b|\|V\|(\bar B_\rho(p)\setminus \bar B_\sigma(p)) \\
		&\quad \leq 2\phi(\rho)\int_{\bar B_\rho(p)}c_b(r)\,\mathrm d\|V\| + \frac{1}{4} \int_{\bar B_\rho(p)\setminus \bar B_\sigma(p)}|H|_g^2\,\mathrm d\|V\| \\
		&\qquad -\phi(\sigma)\int_{\bar B_\sigma(p)}s_b(r)g(\nabla r,H)\,\mathrm d\|V\| +\phi(\rho)\int_{\bar B_\rho(p)}s_b(r)g(\nabla r,H)\,\mathrm d\|V\| \\
		&\qquad + \phi(\sigma)\int_{\bar B_\sigma(p)}s_b(r)g(\nabla r,\eta)\,\mathrm d\|\delta V\|_{\mathrm{sing}} - \phi(\rho)\int_{\bar B_\rho(p)}s_b(r)g(\nabla r,\eta)\,\mathrm d\|\delta V\|_{\mathrm{sing}} \\
		&\qquad +\int_{\bar B_\rho(p)\setminus \bar B_\sigma(p)}\phi(r)s_b(r)g(\nabla r,\eta)\,\mathrm d\|\delta V\|_{\mathrm{sing}} 
	\end{align*}
	for almost all $0<\sigma<\rho<\rho_0$.
\end{lemma}

\begin{proof}
	Given any non-negative smooth function $\varphi:\mathbf R \to\mathbf R$ whose  support is contained in the open interval $(-\infty,\rho_0)$, we define the vector field $X:=\varphi(r)s_b(r)\nabla r$. Write the metric in polar coordinates $g = \mathrm dr \otimes \mathrm dr + g_r$ on $U$ and use Lemma~\ref{lem:pre:Rauch} to estimate
	\begin{equation*}
		\nabla\bigl(s_b(r)\nabla r\bigr) = c_b(r)\mathrm dr\otimes\mathrm dr +s_b(r)\nabla^2 r \geq c_b(r)\mathrm dr\otimes\mathrm dr +s_b(r)\frac{c_b(r)}{s_b(r)}g_r\geq c_b(r) g.
	\end{equation*}
	Hence,
	\begin{equation*}
		\nabla X\geq \varphi'(r)s_b(r)\mathrm dr\otimes\mathrm dr + \varphi(r)c_b(r)g
	\end{equation*}
	which implies
	\begin{equation*}
		\Div_TX \geq \varphi'(r)s_b(r)|\nabla^Tr|_g^2 + 2\varphi(r)c_b(r)  
	\end{equation*}
	for all $T\in\mathbf G_2(TU)$, where $\nabla^Tr$ denotes the orthogonal projection of $\nabla r$ onto $T$. Writing
	\begin{equation*}
		\nabla^\bot r:\mathbf G_2(N)\to TN,\qquad (\nabla^\bot r)(x,T) = (\nabla r) (x)- (\nabla^Tr)(x),
	\end{equation*}
	and testing the first variation equation (see \eqref{eq:intro:var:mean_curvature}) with $X$, we infer 
	\begin{equation}\label{eq:mon-ine:neg-cur:smooth_case}
		\begin{split} 
			&2\int_N\varphi(r)c_b(r)\,\mathrm d\|V\| + \int_{\mathbf G_2(N)}\varphi'(r)s_b(r)(1-|\nabla^\bot r|_g^2)\,\mathrm dV \\
			&\qquad\leq -\int_N\varphi(r)s_b(r)g(\nabla r,H)\,\mathrm d\|V\|+\int_N\varphi(r)s_b(r)g(\nabla r,\eta)\,\mathrm d\|\delta V\|_{\mathrm{sing}}.
		\end{split}
	\end{equation}
	Notice that the function $t\mapsto \|V\|\bar B_t(p)$ is continuous at $t_0$ if and only if $\|V\|(\{r=t_0\}) = 0$. Since the function $t\mapsto \|V\|\bar B_t(p)$ is non-decreasing, it can only have countably many discontinuity points. Choose $0<\sigma<\rho<\rho_0$ to be continuity points.
	Define the non-increasing Lipschitz function
	\begin{equation*}
		\phi_\sigma:(0,\infty)\to \mathbf R,\qquad \phi_\sigma(t) = \phi(\max\{t,\sigma\}) 
	\end{equation*}
	and let $\varphi$ approach $(\phi_\sigma(\cdot) - \phi(\rho))_+$, where $(\cdot)_+:=\max\{\cdot,0\}$. Then, by the dominated convergence theorem, \eqref{eq:mon-ine:neg-cur:smooth_case} becomes
	\begin{equation}\label{eq:mon-ine:neg-cur:nonsmooth_version}
		\begin{split} 
			&2\phi(\sigma)\int_{\bar B_\sigma(p)}c_b(r)\,\mathrm d\|V\| + 2 \int_{\bar B_\rho(p)\setminus \bar B_\sigma(p)}\phi(r)c_b(r)\,\mathrm d\|V\|\\
			&\leq 2\phi(\rho)\int_{\bar B_\rho(p)}c_b(r)\,\mathrm d\|V\|- \int_{\pi^{-1}[\bar B_\rho(p)\setminus \bar B_\sigma(p)]}\phi'(r)s_b(r)[1 - |\nabla^\bot r|^2_g]\,\mathrm d V \\
			&-\phi(\sigma)\int_{\bar B_\sigma(p)}s_b(r)g(\nabla r,H)\,\mathrm d\|V\| -\int_{\bar B_\rho(p)\setminus \bar B_\sigma(p)}\phi(r)s_b(r)g(\nabla r,H)\,\mathrm d\|V\|\\
			& +\phi(\rho)\int_{\bar B_\rho(p)}s_b(r)g(\nabla r,H)\,\mathrm d\|V\| + \phi(\sigma)\int_{\bar B_\sigma(p)}s_b(r)g(\nabla r,\eta)\,\mathrm d\|\delta V\|_{\mathrm{sing}}  \\
			& +\int_{\bar B_\rho(p)\setminus \bar B_\sigma(p)}\phi(r)s_b(r)g(\nabla r,\eta)\,\mathrm d\|\delta V\|_{\mathrm{sing}} - \phi(\rho)\int_{\bar B_\rho(p)}s_b(r)g(\nabla r,\eta)\,\mathrm d\|\delta V\|_{\mathrm{sing}}
		\end{split}
	\end{equation}
	where $\pi:\mathbf G_2(N)\to N$ is the canonical projection. We compute
	\begin{equation}\label{eq:mon-ine:neg-cur:trigonometric_identity}
		2\phi c_b + \phi's_b = \frac{|b|}{(c_b - 1)^2}\left[2c_b^2 - 2c_b - c_b's_b\right] = |b|
	\end{equation}
	as well as $\phi'(r)s_b(r) = -(\phi(r)s_b(r))^2$, and
	\begin{equation}\label{eq:mon-ine:neg-cur:Pythagoras}
		\phi'(r)s_b(r)|\nabla^\bot r|_g^2 - \phi(r)s_b(r)g(\nabla r,H) = -\Bigl|\phi(r)s_b(r)\nabla^\bot r + \frac{1}{2}H\Bigr|_g^2 + \frac{1}{4}|H|_g^2.
	\end{equation}
	Putting \eqref{eq:mon-ine:neg-cur:trigonometric_identity} and \eqref{eq:mon-ine:neg-cur:Pythagoras} into \eqref{eq:mon-ine:neg-cur:nonsmooth_version} and neglecting negative terms on the right hand side implies the conclusion.
\end{proof}

\subsection{Proof of Theorem~\ref{thm:Li-Yau_inequality}}\label{subsec:Li-Yau}
If $b\geq0$, the theorem is a consequence of Lemma~\ref{lem:mon-ine:positive_curvature} in combination with \eqref{eq:rem:mon-ine:pos-cur}. Indeed, by Theorem~\ref{thm:mon-ine:existence_density}, we can multiply the inequality with $4$, let $\sigma \to 0+$ and let $\rho\to \infty$ to conclude the positive curvature case.

Now suppose that $b<0$. We are going to determine the limits in Lemma~\ref{lemma:mon-ine:negative_curvature} as $\sigma\to0+$ and $\rho\to\infty$. Using L'H\^ospital's rule twice, one readily verifies
\begin{equation*}
	\frac{\sigma^2}{\cosh(\sqrt{|b|}\sigma) - 1}\to \frac{2}{|b|}\qquad\text{as $\sigma\to0+$}.
\end{equation*}
Therefore, by Theorem~\ref{thm:mon-ine:existence_density},
\begin{align*}
	&2\phi(\sigma)\int_{\bar B_\sigma(p)}c_b(r)\,\mathrm d\|V\| = 2\pi |b|\frac{\sigma^2}{\cosh(\sqrt{|b|}\sigma) - 1}\frac{1}{\pi\sigma^2}\int_{\bar B_\sigma(p)}\cosh(\sqrt{|b|}r)\,\mathrm d\|V\|\\
	&\qquad \to 4\pi\mathbf \Theta^2(\|V\|,p).
\end{align*}
Similarly, by L'H\^ospital's rule,
\begin{equation*}
	\frac{\sqrt{|b|\pi}\sinh(\sqrt{|b|}\sigma)\sigma}{\cosh(\sqrt{|b|}\sigma)-1}\to 2\sqrt{\pi}\qquad\text{as $\sigma \to0+$}.
\end{equation*}
Hence, by H\"older's inequality and square integrability of the generalised mean curvature,
\begin{align*}
	&\phi(\sigma)\int_{\bar B_\sigma(p)}s_b(r)g(\nabla r,H)\,\mathrm d\|V\|\\
	&\qquad\leq \frac{\sqrt{|b|\pi}\sinh(\sqrt{|b|\sigma})\sigma}{\cosh(\sqrt{|b|}\sigma)-1}\left(\frac{\|V\|\bar B_\sigma(p)}{\pi\sigma^2}\right)^{1/2}\left(\int_{\bar B_\sigma(p)}|H|_g^2\,\mathrm d\|V\|\right)^{1/2}\\
	&\qquad\qquad\to \left(4\pi\mathbf \Theta^2(\|V\|,p)\right)^{1/2}\limsup_{\sigma\to0+}\left(\int_{\bar B_\sigma(p)}|H|_g^2\,\mathrm d\|V\|\right)^{1/2} = 0.
\end{align*}
All the other limits can be easily determined using that $\spt\|V\|$ is compact and using that $p\notin \spt\|\delta V\|_{\mathrm{sing}}$.\qed

\section{Diameter bounds}
In this Section, we provide the lower diameter bounds for varifolds (Lemma~\ref{lem:dia-pin:lower_diameter_bound} and Lemma~\ref{lem:dia-pin:asymptotically_negatively_curved}) that are needed to prove Theorem~\ref{thm:low_dia_bds_min}, Theorem~\ref{thm:lower_dia_bds_min_asypmt_neg}, and Theorem~\ref{thm:diameter_pinching}. To prove Lemma~\ref{lem:dia-pin:asymptotically_negatively_curved} we will need  the Hessian comparison theorem of the distance function for asymptotically non-positively curved manifolds (Lemma~\ref{lem:dia-pin:comparison}). At the end of this section, we will prove Theorem~\ref{thm:diameter_pinching}.

The following lemma is a direct combination of the representation formula for the first variation, Lemma~\ref{lem:pre:representing_first_variation}, with Rauch's comparison theorem, Lemma~\ref{lem:pre:Rauch_application}.
\begin{lemma} \label{lem:dia-pin:lower_diameter_bound}
	Suppose $m, n$ are positive integers, $m\leq n$, $N$ is a complete $n$-dimensional Riemannian manifold,  $p\in N$, $b>0$, $0<\rho<\min\{i_p(N),\frac{\pi}{2\sqrt{b}}\}$, the sectional curvature satisfies $\sup_{B_\rho(p)}K\leq b$, and $V\in \mathbf V_m(N)$ satisfies $\spt \|V\| \subset B_\rho(p)$. 
	
	Then, there holds
	\begin{equation*}
		m\int_N \sqrt{b}r\cot(\sqrt{b}r)\,\mathrm d\|V\| \leq \rho \|\delta V\|(N).
	\end{equation*}
\end{lemma}

\begin{lemma}[\protect{See \cite{MR3237065}}]\label{lem:dia-pin:comparison}
	Suppose $(N,g)$ is a complete Riemannian manifold, $p\in N$, $r = d(p,\cdot)$, $\cut(p)$ denotes the cut locus of $p$ in $N$, $D(p) = N\setminus(\cut(p)\cup \{p\})$, $0<b\leq1/4$, and the radial curvature $K_r$ satisfies
	\begin{equation*}
	K_r \leq \frac{b}{r^2}\qquad \text{on $D(p)$}.
	\end{equation*}
	
	Then, the Hessian $\nabla^2r$ of $r$ can be bounded below by 
	\begin{equation*}
	\nabla^2r\geq \frac{1 + \sqrt{1-4b}}{2r}\Bigl(g - \mathrm dr \otimes \mathrm dr\Bigr)\qquad\text{on $D(p)$}.
	\end{equation*}
\end{lemma}
 
\begin{lemma} \label{lem:dia-pin:asymptotically_negatively_curved}
	Suppose $m,n$ are positive integers, $m\leq n$, $N$ is a complete $n$-dimensional Riemannian manifold, $p\in N$, $r = d(p,\cdot)$, $\cut(p)$ denotes the cut locus of $p$ in $N$, $D(p) = N\setminus(\cut(p)\cup \{p\})$, $0<b\leq1/4$, the radial curvature $K_r$ satisfies $K_r \leq \frac{b}{r^2}$ on $D(p)$, $\rho>0$, $V\in \mathbf V_m(N)$, $\spt \|V\|\subset B_\rho(p)$, and $\|V\|(\cut(p))=0$.
	
	Then, there holds
	\begin{equation*}
	\|V\|(N) \leq \frac{2\rho}{m(1 + \sqrt{1 - 4b})} \|\delta V\|(N).
	\end{equation*}
\end{lemma}

\begin{proof}
	Using Lemma~\ref{lem:dia-pin:comparison}, we compute
	\begin{align*}
		\nabla(r\nabla r) &= \mathrm dr\otimes \mathrm dr + r\nabla^2r \\
		&\geq \frac{1+\sqrt{1-4b}}{2}\mathrm dr\otimes \mathrm dr + \frac{1+\sqrt{1-4b}}{2}\left(g - \mathrm dr\otimes\mathrm dr\right) = \frac{1+\sqrt{1-4b}}{2}g
	\end{align*}
	and thus
	\begin{equation*}
		\Div_T(r\nabla r) \geq m\frac{1+\sqrt{1-4b}}{2}
	\end{equation*}
	for all $T\in\mathbf G_m(TD(p))$. Now we can test the first variation formula~\eqref{eq:intro:var:mean_curvature} with the vector field $r\nabla r$ and use Lemma~\ref{lem:pre:representing_first_variation} to conclude the proof.
\end{proof}

\subsection{Proof of Theorem~\ref{thm:diameter_pinching}}\label{subsec:proof_diameter_pinching}
The following proof is based on \textsc{Simon}~\cite[Lemma~1.1]{MR1243525}.

Let $0<\rho<\rho_0:=\min\{i,\frac{\pi}{2\sqrt{b}}\}$. Given any $p\in N$, we notice that $\|V\|B_\sigma(p) = \|V\|\bar B_\sigma(p)$ for all but countably many $\sigma>0$. Hence, we can combine \eqref{eq:thm:dia-pin:lower_density_bound} with Theorem~\ref{thm:mon-ine:existence_density}, \eqref{eq:rem:mon-ine:pos-cur}, and Lemma~\ref{lem:mon-ine:positive_curvature} to deduce
\begin{align*}
	\pi &\leq \frac{\|V\|B_\rho(p)}{\rho^2} + \frac{1}{16}\int_{B_\rho(p)}|H|^2_g\,\mathrm d\|V\| + b\|V\|B_\rho(p)\\
	&\qquad + \frac{1}{4}\left(\frac{\|V\|B_\rho(p)}{\rho^2} + \int_{B_\rho(p)}|H|_g^2\,\mathrm d\|V\|\right)
\end{align*}
and thus
\begin{equation}\label{eq:proof:dia-pin:monotonicity_inequality}
	\pi \leq (2 + b\rho^2)\frac{\|V\|B_\rho(p)}{\rho^2} + \frac{1}{2}\int_{B_\rho(p)}|H|_g^2\,\mathrm d\|V\|.
\end{equation}
Define
\begin{equation}\label{eq:proof:dia-pin:definition_area_bound}
	C_{\eqref{eq:proof:dia-pin:definition_area_bound}}(i,b):=\min\Bigl\{\frac{\rho_0^2}{2 + b\rho_0^2},\frac{\pi^3}{9b}\Bigr\}\qquad \text{for $\rho_0 = \min\Bigl\{i,\frac{\pi}{2\sqrt{b}}\Bigr\}$},
\end{equation}
where $\frac{1}{0}$ is interpreted $\infty$ and $0\cdot\infty$ is interpreted $0$. For the rest of the proof we assume that $\|V\|(N)\leq C_{\eqref{eq:proof:dia-pin:definition_area_bound}}(i,b)$. Then, it follows from~\eqref{eq:proof:dia-pin:monotonicity_inequality} that
\begin{equation}\label{eq:proof:dia-pin:lower_bound_Willmore}
	\pi\leq \int_{N}|H|^2_g\,\mathrm d\|V\|
\end{equation}
which proves \eqref{eq:thm:dia-pin:lower_bound_Willmore}. Now, pick
\begin{equation*}
	\rho :=\frac{1}{2}\sqrt{\|V\|(N)\left/\int |H|_g^2\,\mathrm d\|V\|\right.}.
\end{equation*}
From \eqref{eq:proof:dia-pin:definition_area_bound} and \eqref{eq:proof:dia-pin:lower_bound_Willmore} we deduce $2\rho \leq \min\{i,\frac{\pi}{3\sqrt{b}}\}$. In particular, $\rho<\rho_0$. If also $d_{\mathrm{ext}}(\spt\|V\|)<\min\{i,\frac{\pi}{3\sqrt{b}}\}$, then we can combine \eqref{eq:pre:rem:Rauch:lower_bound_pithird} with Lemma~\ref{lem:dia-pin:lower_diameter_bound} and use H\"older's inequality to infer
\begin{equation*}
	\|V\|(N)\leq d_{\mathrm{ext}}(\spt\|V\|)\|\delta V\|\leq d_{\mathrm{ext}}(\spt\|V\|) \sqrt{\|V\|(N)\int_N|H|^2_g\,\mathrm d\|V\|}
\end{equation*}
which implies $2\rho \leq d_{\mathrm{ext}}(\spt\|V\|)$ and which proves~\eqref{eq:thm:dia-pin:lower_diameter_bound}. In any case, we are guaranteed that $2\rho \leq d_{\mathrm{ext}}(\spt\|V\|)$. Now, choose a point $p\in\spt\|V\|$ such that $d_{\mathrm{ext}}(\spt\|V\|) = \max_{q\in\spt\|V\|}d(p,q)$ and let $\nu$ be the integer with 
\begin{equation*}
	\frac{d_{\mathrm{ext}}(\spt\|V\|)}{\rho}-1<\nu\leq\frac{d_{\mathrm{ext}}(\spt\|V\|)}{\rho}. 
\end{equation*}
Let $r:=d(p,\cdot)$ and for each $j=1,\ldots,\nu$ choose $p_j\in \spt\|V\|\cap\{r=\rho j\}$. Then, with $p_0:=p$, the balls $B_{\rho/2}(p_0),\ldots,B_{\rho/2}(p_\nu)$ are pairwise disjoint. Thus, summing over $j=0,\ldots,\nu$ in \eqref{eq:proof:dia-pin:monotonicity_inequality}, we infer
\begin{equation*}
	(\nu + 1)\pi \leq \frac{2}{\rho^2}\|V\|(N) + \frac{1}{2}\int_N|H|_g^2\,\mathrm d\|V\| + b\|V\|(N).
\end{equation*} 
Hence, using $\nu + 1\geq d_{\mathrm{ext}}(\spt\|V\|)/\rho$ and using \eqref{eq:proof:dia-pin:lower_bound_Willmore}, we deduce \eqref{eq:thm:dia-pin:upper_diameter_bound} which concludes the proof. \qed 

\section{Sobolev and isoperimetric inequalities inequalities}
In this section we will prove the Sobolev inequality (see Theorem~\ref{thm:Sobolev_inequality}) and the isoperimetric inequality (see Corollary~\ref{cor:sob-ine}). First, we will need the following Lemma.

\begin{lemma}[\protect{See \cite[Lemma 18.7]{MR756417} or \cite[Lemma 6.3]{scharrer:MSc}}]\label{lem:sob-ine}
	Suppose $m$ is a positive integer, $f,g$ are real valued functions on the interval $(0,\infty)$, $f$ is bounded and non-decreasing,
	\begin{equation*}
		1\leq\limsup_{t\to0+}\frac{f(t)}{t^m},
	\end{equation*}	
	and 
	\begin{equation*}
		\frac{f(\sigma)}{\sigma^{m}}\leq \frac{f(\rho)}{\rho^m} + \int_\sigma^\rho\frac{1}{t^m}g(t)\,\mathrm dt
	\end{equation*}
	for all $0<\sigma<\rho\leq\rho_0:=(2^{m+1}\lim_{t\to\infty}f(t))^{1/m}$.
	
	Then, there exists $0<\rho\leq \rho_0$ such that
	\begin{equation*}
		f(5\rho) < \frac{5^m}{2} \rho_0 g(\rho).
	\end{equation*}
\end{lemma}

\subsection{Proof of Theorem~\ref{thm:Sobolev_inequality}}\label{subsec:proof:Sobolev_inequality}

The following proof is an adaptation of \cite[Theorem~7.1]{MR0307015} and \cite[Theorem 6.5]{scharrer:MSc}.

Let $h\leq1$ be a non-negative compactly supported $C^1$ function on $N$ and define the varifold $V_h\in\mathbf V_m(N)$ by letting
\begin{equation*}
	V_h(k) = \int_{\mathbf G_m(N)}k(x,T)h(x)\,\mathrm d V(x,T)
\end{equation*}
for all compactly supported continuous functions $k$ on $\mathbf G_m(N)$. Given any $X\in\mathscr X(N)$, we compute for $T\in\mathbf G_m(TN)$ (see \cite[Lemma 3.2(ii)]{MR365424})
\begin{equation*}
	\Div_T(hX) = h\Div_TX + g(X,\nabla^Th)
\end{equation*}
and thus, by Lemma~\ref{lem:pre:representing_first_variation},
\begin{align*}
	\delta V_h(X) &= \int_Nhg(X,\eta)\,\mathrm\|\delta V\| - \int_{\mathbf G_m(N)}g(X,\nabla^\top h)\,\mathrm d V.
\end{align*}
Given any open set $U\subset N$, it follows
\begin{equation}\label{eq:proof:Sobolev:weighted_first_variation}
	\|\delta V_h\|(U)\leq \int_Uh\,\mathrm d\|\delta V\| + \int_{\mathbf \pi^{-1}[U]}|\nabla^\top h|_g\,\mathrm dV
\end{equation}
where $\pi:\mathbf G_m(N)\to N$ is the canonical projection. With the usual approximation from above (see \cite[2.1.3 (5)]{MR41:1976}), one can see that the inequality remains valid for $U$ replaced with any closed set. In particular, $\|\delta V_h\|$ is a Radon measure. Hence, we can apply Lemma~\ref{lem:mon-ine:m-dimensional} in combination with Lemma~\ref{lem:pre:representing_first_variation}, Lemma~\ref{lem:pre:Rauch_application}, and \eqref{eq:mon-ine:rem:m-dim} to deduce
\begin{equation*}
	\frac{\|V_h\|\bar B_\sigma(p)}{\sigma^m} \leq \frac{\|V_h\|\bar B_\rho(p)}{\rho^m} + \int_\sigma^\rho\frac{1}{t^m}\left(\|\delta V_h\|\bar B_t(p) + m\sqrt{b}\|V_h\|\bar B_t(p)\right)\,\mathrm dt
\end{equation*}
for all $p\in N$ and $0<\sigma<\rho<\min\{i(\spt\|V\|),\frac{\pi}{2\sqrt{b}}\}$.
Define the constant
\begin{equation}\label{eq:Sob-ine:constant}
	C_\eqref{eq:Sob-ine:constant}(m):=\frac{5^m2^{1/m}}{\boldsymbol\alpha(m)^{1/m}}.	
\end{equation}
Given any $p\in\spt\|V\|$ with $\mathbf \Theta^{*m}(\|V_h\|,p)\geq1$, Lemma~\ref{lem:sob-ine} applied with $f(t) = \boldsymbol\alpha(m)^{-1}\|V_h\|\bar B_t(p)$ and 
\begin{equation*}
	g(t) = \boldsymbol\alpha(m)^{-1}\left(\|\delta V_h\|\bar B_t(p) + m\sqrt{b}\|V_h\|\bar B_t(p)\right)
\end{equation*}
implies by \eqref{eq:proof:Sobolev:weighted_first_variation}
\begin{align*}
	&\int_{\bar B_{5\rho}(p)} h\,\mathrm d\|V\|\leq 	C_\eqref{eq:Sob-ine:constant}(m)\left(\int h\,\mathrm d\|V\|\right)^{1/m}\int_{\bar B_\rho(p)}h\,\mathrm d\|\delta V\|\\ 
	&\quad + C_\eqref{eq:Sob-ine:constant}(m)\left(\int h\,\mathrm d\|V\|\right)^{1/m}\left(m\sqrt{b}\int_{\bar B_\rho(p)}h\,\mathrm d\|V\|+ \int_{\pi^{-1}[\bar B_\rho(p)]}|\nabla^\top h|_g\,\mathrm dV\right). 
\end{align*}
Now, Vitali's covering theorem (see \cite[2.8.5,\,6,\,8]{MR41:1976}) implies the assertion. \qed

\subsection{Proof of Corollary~\ref{cor:sob-ine}} \label{subsec:proof:isoperimetri_inequality}
First, we apply Theorem~\ref{thm:Sobolev_inequality} with $h$ approaching the constant function $1$ from below, to deduce
\begin{equation*}
	\|V\|(N)\leq C_{\eqref{eq:Sob-ine:constant}}\|V\|(N)^{1/m}\left(\|\delta V\|(N) + m\sqrt{b}\|V\|(N)\right).
\end{equation*}
Applying this inequality again on the right hand side, we inductively infer for all positive integers $k\geq1$,
\begin{align*}
	\|V\|(N)&\leq C_{\eqref{eq:Sob-ine:constant}}\|V\|(N)^{1/m}\|\delta V\|(N)\sum_{j=0}^{k-1}\left(m\sqrt{b}C_{\eqref{eq:Sob-ine:constant}}\|V\|(N)^{1/m}\right)^j \\
	&\qquad+ \left(m\sqrt{b}C_{\eqref{eq:Sob-ine:constant}}\|V\|(N)^{1/m}\right)^k\|V\|(N) \\
	&\leq C_{\eqref{eq:Sob-ine:constant}}\|V\|(N)^{1/m}\|\delta V\|(N)\sum_{j=0}^\infty \left(m\sqrt{b}C_{\eqref{eq:Sob-ine:constant}}\|V\|(N)^{1/m}\right)^j\\
	&\leq 2C_{\eqref{eq:Sob-ine:constant}}\|V\|(N)^{1/m}\|\delta V\|(N) 
\end{align*}
provided that $m\sqrt{b}C_{\eqref{eq:Sob-ine:constant}}\|V\|(N)^{1/m}\leq\frac{1}{2}$.

\end{document}